\documentclass[
a4paper,
12pt,
DIV=14,`1234
toc=bibliography,
headsepline,
parskip=half-,
abstract=true,
]{scrartcl}

\usepackage{amsmath,amsfonts,amsthm,amssymb,nicefrac}
\usepackage{graphicx,color,url}
\usepackage{xcolor}
\usepackage{accents}
\usepackage{url,hyperref}
\usepackage{enumerate}
\usepackage{comment}
\usepackage{scrlayer-scrpage}

\newtheorem{theorem}{Theorem}
\newtheorem{lemma}[theorem]{Lemma}
\newtheorem{corollary}[theorem]{Corollary}
\newtheorem{proposition}[theorem]{Proposition}
\theoremstyle{definition}

\numberwithin{equation}{section}\numberwithin{theorem}{section}
\allowdisplaybreaks
\newcounter{stepctr}
{\end{list}}

\def\XXint#1#2#3{{\setbox0=\hbox{$#1{#2#3}{\int}$}
 \vcenter{\hbox{$#2#3$}}\kern-.5\wd0}}

\newcommand{\circo}{\accentset{\circ}}
\newcommand{\ra}{\rangle}
\newcommand{\la}{\langle}
\newcommand{\e}{\varepsilon}

\providecommand{\titlemacro}{{High Codimension MCF in $\mathbb{C}P^n$}}
\title{\titlemacro}
\author{Artemis A. Vogiatzi}
\ohead{\titlemacro}
\ihead{A.A.Vogiatzi}
\setkomafont{pageheadfoot}{\footnotesize}

\date{}
\begin{document}
\title{High Codimension Mean Curvature Flow in the Complex Projective Space}
\maketitle
\begin{abstract}
We study the mean curvature flow of smooth, compact submanifolds of high codimension with quadratic pinching in $\mathbb{C}P^k$.
We establish a codimension estimate, where in regions of high curvature, the submanifold becomes approximately codimension one, in a quantifiable way. 
Under a cylindrical type pinching, this limiting flow is weakly convex and moves by translation or is a self shrinker. These estimates allow us to analyse the behaviour of the flow near singularities and establish the existence of the limiting flow. With a decay estimate, the rescaling converges smoothly to a totally geodesic limit in infinite time, without using Stampacchia iteration or integral analysis. 

\end{abstract}
\section{Introduction}\label{sec_introduction}
Let $F_0\colon\mathcal{M}^n \rightarrow \mathbb{C}P^k$ be a smooth immersion of a compact manifold $\mathcal{M}^n$ and codimension $m\ge2$. The mean curvature flow starting from $F_0$ is the following one-parameter family of submanifolds $F\colon \mathcal{M}^n \times[0, T) \rightarrow \mathbb{C}P^k$, such that
	\begin{align*}
	\begin{split}
\left\{
	\begin{array}{rl}
		\partial_t F(p, t) &=H(p, t), \ \ \text{for} \ \ p \in \mathcal{M}, t \in[0, T)   \\
		F(p, 0) &=F_0(p)
	\end{array}
	\right.
\end{split}
	\end{align*}
where $H(p, t)$ denotes the mean curvature vector of $\mathcal{M}_t=F_t(p)=F(p, t)$ at $p$. It is well known this is a system of quasilinear weakly parabolic partial differential equations for $F$. Geometrically, the mean curvature flow is the steepest descent flow for the area functional of a submanifold and hence it is a natural curvature flow.
\\
So far, a lot of work has been done on submanifolds evolving by their mean curvature over the last decades, especially in the case of codimension one. In \cite{Hu84}, Huisken proved that closed convex hypersurfaces in Euclidean space under the mean curvature flow evolve into spherical singularities, using Stampacchia iteration and the Michael--Simons--Sobolev inequality. In \cite{Hu86}, Huisken then generalises this theorem to Riemannian background curvature spaces with strict convexity depending on the geometry of the background space. \\
In \cite{AnBa10}, Andrews and Baker proved convergence to a round point for submanifolds of higher codimension of the Euclidean space under mean curvature flow, satisfying a suitable quadratic pinching condition. In \cite{Liu}, Liu, Xu and Zhao have obtained similar results for submanifolds of hyperbolic spaces and of general Riemannian manifolds. \\
Recently, in \cite{HNAV}, we proved a codimension estimate in high codimension mean curvature flow: under the quadratic pinching condition $|A|^2 \leq c_n|H|^2-d_n, c_n \leq \frac{4}{3 n}, d_n>0$ any singularity model must necessarily be of codimension one, regardless of the codimension of the initial flow. Together with a cylindrical estimate, this led to a classification of all singularity models for pinched submanifolds of high codimension in general Riemannian manifolds. This differs from the classical convergence results of Huisken \cite{Hu84,Hu86}, Andrews-Baker \cite{AnBa10}, and Liu-Xu-Zhao \cite{Liu}, where the entire manifold is shown to contract smoothly to a round point; in our setting, local singularities may occur, and we study them via blow-up and rescaling arguments. Related results in this direction have also been obtained by Naff \cite{Naff,Naff6}. In addition, in the same paper \cite{HNAV}, we established convergence to a codimension one limiting flow for submanifolds of higher codimension in hyperbolic space under a suitable quadratic pinching condition.
\\
A common feature of results on mean curvature flow is the dichotomy between finite time singularity formation and long time convergence to a steady limit. For weak solutions, White \cite{White00} showed that mean convex flows either disappear in finite time or converge to a finite collection of stable minimal submanifolds, while Grayson \cite{Grayson89} proved that an embedded curve in a Riemannian surface either shrinks to a round point or converges smoothly to a geodesic. In the case of submanifolds of the sphere, Huisken \cite{Hu87} and Baker \cite{BakerThesis} made essential use of the positive curvature, which allows precisely this dichotomy: under a pinching condition of the form $|A|^ 2<c|H|^ 2+d$, with $d\ge0$, showed that if the maximal time is infinite, the flow converges smoothly to a totally geodesic hypersurface. This higher dimensional analogue of Grayson's result has been obtained in the sphere setting, by Huisken \cite{Hu87} in codimension one and by Baker \cite{BakerThesis} in higher codimension. These results can be stated together as follows.
\begin{theorem}[\cite{AnBa10},\cite{Hu87}]
Let $\mathcal{M}_0$ be a closed $n$ dimensional submanifold of $\mathbb{S}^{n+k}$, with $n \geq 2$, and suppose that we have on $\mathcal{M}_0$
	\begin{align*}
& |A|^2<\frac{1}{n-1}|H|^2+2, \quad \text { if } n \geq 4 \text {, or } n=3 \text { and } k=1, \\
& |A|^2<\frac{3}{4}|H|^2+\frac{4}{3}, \quad \text { if } n=2 \text { and } k=1 \text {, } \\
& |A|^2<\frac{4}{3 n}|H|^2+\frac{2(n-1)}{3}, \quad \text { if } n=2,3 \text { and } k>1.
	\end{align*}
Then, one of the following holds.
\begin{enumerate}
\item If $T_{\max }<\infty$, then the $\mathcal{M}_t$ 's converge to a round point as $t \rightarrow T_{\max }$,
\item If $T_{\max }=\infty$, then the $\mathcal{M}_t$ 's converge to a smooth totally geodesic hypersurface $\mathcal{M}_{\infty}$, isometric to $\mathbb{S}^n$.
\end{enumerate}
\end{theorem}
In \cite{LaLyNg}, Langford, Lynch and Nguyen developed these results further by allowing a weaker curvature pinching condition. 
While the unique properties of the sphere are pivotal in \cite{Hu87}, \cite{BakerThesis}, and \cite{LaLyNg}, this applies for more general ambient spaces of positive curvature. In \cite{PipSin}, Pipoli and Sinetrari considered the complex projective space and showed that suitably pinched submanifolds evolving by mean curvature flow exhibit similar properties to the ones of the sphere. 
In this paper, we follow Nguyen and the author in \cite{HNAV} in the setting of complex projective space and derive a different behaviour of the flow in infinite time. More specifically, we extend the result of Pipoli and Sinestrari \cite{PipSin} by allowing general constants in the quadratic pinching condition, which permits new singular behaviours besides the round point case, and we avoid using Stampacchia iteration. Related work on invariance of pinching and cylindrical estimates, though only in codimension one, was obtained in \cite{PipSin17}. Using the result of \cite{HNAV}, we then derive our main theorem, stated below.
\begin{theorem}\label{maintheorem}
Let $\mathcal{M}_0$ be a closed submanifold of $\mathbb{C}P^k$ of real dimension $n$ and codimension $m = 2k - n$. Let $\tilde{k}=2$. For $m=1$, we have $n\ge\max \{9, 4\tilde{k}\}$ and for $2\le m<\frac{n-3}{4}$, we have $n\ge\max\{23, 4\tilde{k}\}$. If at every point of $\mathcal{M}_0$ the inequality
	\begin{align}\label{pinching_conditioncpk}
|A|^2< \begin{cases}\frac{1}{n-\tilde{k}}|H|^2+2\tilde{k} & \text { if } m=1, \\ \frac{1}{n-\tilde{k}}|H|^2+\frac{\tilde{k}(n-3-4 m)}{n} & \text { if } 2\le m <  \frac{n-3}{4},\end{cases}
	\end{align}
is satisfied, then the same holds on $\mathcal{M}_t$ for all $0 < t < T$, where $T$ is the maximal time of existence. In addition, one of the following holds.
	\begin{enumerate}
\item If $T< \infty$, then the blow up of $\mathcal{M}_t$ is a codimension one ancient flow as $t \to T$,
\item If $T = \infty$, then $\mathcal{M}_t$ converges to a smooth totally geodesic submanifold as $t \to T$.
	\end{enumerate}
Case (2) can only occur if $n$ is even, and the limit submanifold is isometric to $\mathbb{C}P^{\frac{n}{2}}$.
	\end{theorem}
In case (1), the only possible codimension one ancient flow limits are the shrinking round sphere $\mathbb{S}^n$, the shrinking round cylinder $ \mathrm{S}^{\mathrm{n}-1}\times \mathbb{R}$ and the translating bowl soliton.\\
We call an inequality of the form \eqref{pinching_conditioncpk} above a pinching condition on the second fundamental form.\\
The above statement says in particular that in odd dimensions a submanifold satisfying our assumptions converges to a codimension one solution under the mean curvature flow. This result is not directly proven, as we show that the only alternative to a codimension one limiting flow is the behaviour in (2), but such a behaviour is excluded for odd dimensions since the only totally geodesic submanifolds of $\mathbb{C}P^k$, with small codimension, like in the hypothesis of the theorem above, are isometric to a complex projective space. Full classification on totally geodesic submanifolds of $\mathbb{C}P^k$ can be found in Theorem 3.25 in \cite{Besse}. \\
Nguyen in \cite{HTNsurgery} developed a surgery construction allowing high codimension mean curvature flow with cylindrical pinching to pass through singularities. This generalised the codimension one result of \cite{HuSi09} (see also \cite{HK2}) to high codimension. A key aspect of this surgery procedure is the codimension estimate presented in \cite{Naff}, which shows that near regions of high curvature, singularities become approximately codimension one. Another crucial component is the  cylindrical estimate, which shows that nears regions of high curvature, the submanifold becomes approximately cylindrical of the form $ \mathbb S^{n-1}\times \mathbb R$.  These estimates are essential for the surgery to work and allow us to control the geometry of the submanifold in regions of high curvature.


The outline of the paper is as follows. In section 2, we give all the technical tools needed for our work and set up our notation. In section 3, we give the proof for the preservation of both pinching conditions along the mean curvature flow and we prove the gradient estimate. In section 4, we prove the codimension estimate, the cylindrical estimate and we give a full classification for the singularity models of quartically pinched solutions in high codimension mean curvature flow. Finally, in section 6, we prove  that the submanifold converges smoothly to a totally geodesic limit in infinite time.

\textbf{Acknowledgements} This result will form part of the author's doctoral thesis at Queen Mary University of London. She wishes to thank her supervisor, Dr Huy The Nguyen, for all the helpful discussions and encouragement during the writing of this paper.

\section{Preliminaries}
This section presents the necessary preliminary results and establishes our notation. We primarily follow \cite{PipSin} with our own notation. We use Greek indices for the normal bundle, and Latin ones for the tangent bundle. We use the indices $i,j,k,\dots$ for the submanifold $F(\mathcal{M})$ and indices $a,b,c,\dots$ for the background space. More specifically, we adopt the following convention for indices.
	\begin{align*}
	1\le i,j,k,\ldots\le n, \ 1\le a,b,c,\ldots\le n+m \ \ \text{ and} \ \ 1\le \alpha,\beta,\gamma,\ldots \le m.
	\end{align*}
We choose a local orthonormal frame $\{e_i\}$ of $T_p \mathcal{M}_t$ and $\{\nu_\alpha\}$ for the normal bundle $N_p \mathcal{M}_t$, where
	\begin{align*}
	\nu_1:=\frac{H}{|H|}
	\end{align*} 
is the principal normal direction. This is well defined since in our setting $|H| \neq 0$. We will use two kinds of frames from now on. The one we just mentioned will be called of kind (B1).
 Let $\left\{\omega_\alpha\right\}$ be the dual frame field of $\left\{e_i\right\}$ and $\left\{\nu_\alpha\right\}$. The metric $g$ and the volume form $d \mu$ of $\mathcal{M}$ are $g=\sum_i \omega_i \otimes \omega_i$ and $d \mu=\omega_1 \wedge \ldots \wedge \omega_n$.\\
The manifold $\mathbb{C} P^k$ is a Kähler manifold with complex dimension $k$, equipped with the Fubini-Study metric $g_{F S}$ and with complex structure $J$. The $\mathbb{C} P^k$ can be viewed as a real Riemannian manifold of dimension $2 k$. The curvature tensor and the Levi-Civita connection of $\left(\mathbb{C} P^k, g_{F S}\right)$ are denoted by $\bar{R}$ and $\bar{\nabla}$, respectively. The explicit form of $\bar{R}$ for any tangent vector fields $X, Y$, $Z, W$ is given by:
	\begin{align}\label{RinCPn}
\bar{R}(X, Y, Z, W)&=  g_{F S}(X, Z) g_{F S}(Y, W)-g_{F S}(X, W) g_{F S}(Y, Z) +g_{F S}(X, J Z) g_{F S}(Y, J W)\nonumber\\
&-g_{F S}(X, J W) g_{F S}(Y, J Z)+2 g_{F S}(X, J Y) g_{F S}(Z, J W).
	\end{align}
The sectional curvature $\bar{K}(X, Y)$ of a tangent plane spanned by two orthonormal vector fields $X$ and $Y$ is given by
	\begin{align*}
\bar{K}(X, Y)=1+3 g_{F S}(X, J Y)^2,
	\end{align*}
which implies $1 \leq \bar{K} \leq 4$. The values $\bar{K}=1$ and $\bar{K}=4$ are if and only if $J Y$ is orthogonal (resp. tangent) to $X$. Moreover, $\left(\mathbb{C} P^k, g_{F S}\right)$ is a symmetric space, satisfying $\bar{\nabla} \bar{R}=0$, and it is also an Einstein manifold with Einstein constant $2(k+1)$.\\
Now, consider a closed submanifold $\mathcal{M}$ of $\mathbb{C} P^k$, with the induced metric $g$, curvature tensor $R$, and connection $\nabla$. As before, we use $n$ to represent the dimension of $\mathcal{M}$ and $m=$ $2k-n$ to represent its codimension. \\
We denote by $A$ to be the normal vector valued second fundamental form tensor and denote by $H$ the mean curvature vector which is the trace of the second fundamental form given by $H^\alpha=\operatorname{tr}_g A^\alpha=\sum_i A^\alpha_{ii}$ and $|H|$ is the length. The tracefree second fundamental form $\mathring{A}$  is defined by $\mathring{A}=A-\frac{1}{n}Hg$, whose components are given by $\mathring{A}^\alpha_{ij}=A^\alpha_{ij}-\frac{1}{n}H^\alpha g_{ij}$. Obviously, we have $\sum_i \mathring{A}^\alpha_{ii}=0$. In particular, the squared length satisfies 
	\begin{align*}
	|\mathring{A}|^2=|A|^2-\frac{1}{n}|H|^2.
	\end{align*}
We denote by $A^-$ the second fundamental form tensor orthogonal to the principal direction and $h$ to be the tensor valued second fundamental form in the principal direction, that is 
	\begin{align*}
	h_{ij}=\frac{\langle A_{ij},H\rangle}{|H|}. 
	\end{align*}
Therefore, we have $A=A^-+h\nu_1$ and $|A|^2=|A^-|^2+|h|^2$. We denote $\mathring{h}$ to be the traceless part of the second fundamental form in the principal direction. \\
Given a connection $\nabla$ on $A$, from the definition of $A^-$, it is natural to define the connection $\hat{\nabla}^\bot$ acting on $A^-$, by
	\begin{align*}
	\hat{\nabla}^\bot_i A^-_{jk}:=\nabla^\bot_i A^-_{jk}-\langle \nabla^\bot_i A^-_{jk},\nu_1\rangle\nu_1.
	\end{align*}
Here and in the following, $\nabla^{\perp}$ always denotes the normal connection. When no confusion arises we sometimes omit the superscript and simply write $\nabla$. The traceless second fundamental form can be rewritten as $\mathring{A}=\sum_{\alpha} \mathring{A}^\alpha \nu_{\alpha}$, where
		\begin{align*}
	\begin{cases} H=\operatorname{tr} A=\operatorname{tr}h,\ \ \ \ \alpha= 1 &\\ H^\alpha=\operatorname{tr} A^\alpha=0, \ \ \ \ \ \ \alpha \geq 2\end{cases}
\ \ \text{and} \ \ \ \
	\begin{cases}\mathring{h}=h-\frac{|H|}{n} Id, \ \ \ \ \alpha= 1 &\\ \mathring{A}^-=A^\alpha,\ \ \ \ \ \ \ \ \ \ \ \ \ \ \alpha \geq 2.\end{cases}
	\end{align*}
When using a basis of this kind, we adopt the following notation:
	\begin{align*}
	 |A^-|^2=|\mathring{A}^-|^2:=\sum_{\alpha\ge2}|\mathring{A}^{\alpha}|^2.
	\end{align*}
Let $R_{ijkl}=g\left(R(e_i,e_j)e_k,e_l\right),  \ \ \bar{R}_{abcd}=\langle\bar{R}(e_a,e_b)e_c,e_d\rangle \ \ \text{and} \ \ R_{ij\alpha\beta}^\bot =\langle R^\bot(e_i,e_j)\nu_\alpha,\nu_\beta\rangle.$ If $\mathcal{M}$ is a hypersurface, then the mean curvature vector is a multiple of the unit normal vector $\nu$ and satisfies
	\begin{align*}
H=-\left(\lambda_1+\cdots+\lambda_n\right) \nu,
	\end{align*}
where $\lambda_1 \leq \ldots \leq \lambda_n$ are the principal curvatures. In addition, we have $|A|^2=\lambda_1^2+\cdots+\lambda_n^2$ and
	\begin{align*}
|\mathring{A}|^2=|A|^2-\frac{1}{n}|H|^2=\frac{1}{n} \sum_{i<j}\left(\lambda_i-\lambda_j\right)^2,
	\end{align*}
so that smallness of $|\mathring{A}|^2$ implies that the curvatures are close to each other.
The mean curvature vector $H \in \Gamma(N\mathcal{M})$ can be written in the various forms
	\begin{align*}
H=\operatorname{tr}_g h=g^{i j} h_{i j}=h_i{ }^i=h_{i i}=g^{i j} h_{i j}{ }^\alpha \nu_\alpha=h_{i i \alpha} \nu_\alpha .
	\end{align*}
Similarly, we write $|h|^2=g^{i k} g^{j l} g_{\alpha \beta}^N h_{i j}{ }^\alpha h_{k l}{ }^\beta=h_{i j \alpha} h_{i j \alpha}$. In higher codimension, the Ricci equation can be written as follows.
	\begin{align}\label{Riccieq}
R^\bot_{ij\alpha\beta}=\bar{R}_{ij\alpha\beta}+\sum_{p}\left(A^\alpha_{ip}A^\beta_{jp}-A^\beta_{ip}A^\alpha_{jp}\right).
	\end{align}
The evolution equations of the length and squared length of the mean curvature and second fundamental form in a general Riemannian space have been computed in \cite{AnBa10} and \cite{BakerThesis}. In our case, these equations become simpler due to the symmetry of the ambient manifold.
\begin{proposition}\label{eqn|A|^2|H|^2}
With the summation convention, the evolution equations of $A_{ij}, H$ are
	\begin{align}\label{7eqn_A}
	&\left(\nabla_t-\Delta\right)A_{ij}=\sum_{p,q}\langle A_{ij},A_{pq}\rangle A_{pq}+\sum_{p,q}\langle A_{iq},A_{pq}\rangle A_{pj}+\sum_{p,q}\langle A_{jq},A_{pq}\rangle A_{pi}-2\sum_{p,q}\langle A_{ip},A_{jq}\rangle A_{pq}\nonumber\\
	&\ \quad\quad\quad\quad\quad+2\sum_{p,q}\bar{R}_{ipjq}A_{pq}-\sum_{k,p}\bar{R}_{kjkp} A_{pi}-\sum_{k,p}\bar{R}_{kikp}A_{pj}+\sum_{k,\alpha,\beta}A^\alpha_{ij}\bar{R}_{k\alpha k\beta} \nu_\beta\nonumber\\
	&\ \quad\quad\quad\quad\quad-2\sum_{p,\alpha,\beta}A^\alpha_{jp}\bar{R}_{ip\alpha\beta}\nu_\beta-2\sum_{p,\alpha,\beta}A^\alpha_{ip}\bar{R}_{jp\alpha\beta}\nu_\beta,\\
 &\left(\nabla_t-\Delta\right)H=\sum_{p,q}\langle H,A_{pq}\rangle A_{pq}+\sum_{k,\alpha,\beta}H^\alpha\bar{R}_{k\alpha k\beta}\nu_\beta.\nonumber
	\end{align}
\end{proposition}
\begin{lemma}\label{eqn_AHCP}
Let us consider a family of immersions $F\colon \mathcal{M}^n\times [0,T) \to \mathbb{C}P^{k}$ moving by mean curvature flow. Then, we have the following evolution equations
	\begin{align}\label{eqn_|A|^2}
	&\left(\partial_t-\Delta\right) |A|^2=-2|\nabla A|^2+2\sum_{\alpha,\beta}\left(\sum_{i,j} A^\alpha_{ij} A_{ij}^\beta \right)^2+2\sum_{i,j,\alpha,\beta}\left(\sum_p\left(A^\alpha_{ip}A_{jp}^\beta -A^\alpha_{jp}A_{ip}^\beta \right)\right)^2\nonumber\\
	& \ \ \quad\quad\quad\quad\quad+4\sum_{i,j,p,q}\bar{R}_{ipjq}\left(\sum_{\alpha} A^\alpha_{pq}A^\alpha_{ij}\right)-4\sum_{j,k,p}\bar{R}_{kjkp}\left(\sum_{i,\alpha} A^\alpha_{pi}A^\alpha_{ij}\right)\nonumber\\
	& \ \ \quad\quad\quad\quad\quad+2\sum_{k,\alpha,\beta}\bar{R}_{k\alpha k\beta}\left(\sum_{i,j} A^\alpha_{ij}A_{ij}^\beta \right)-8\sum_{j,p,\alpha,\beta}\bar{R}_{jp\alpha\beta}\left(\sum_iA^\alpha_{ip}A_{ij}^\beta \right),\\
\label{eqn_|H|^2} &\left(\partial_t-\Delta\right) |H|^2=-2|\nabla H|^2+2\sum_{i,j}\left(\sum_{\alpha} H^\alpha A^\alpha_{ij}\right)^2+2\sum_{k,\alpha,\beta}\bar{R}_{k\alpha k\beta} H^\alpha H^\beta.
	\end{align}
\end{lemma}
By Berger's inequality in \cite{Goldberg} and recalling \eqref{pinching_conditioncpk}, we have
	\begin{align}\label{Berger}
	\begin{split}
&|\bar{R}_{acbc}|\le\frac{1}{2}(K_1+K_2), \ \ \text{ for} \ a\neq b,\\
	&|\bar{R}_{abcd}|\le\frac{2}{3}(K_1+K_2), \ \ \text{for all distinct indices} \ a,b,c,d.
	\end{split}
	\end{align}
When the codimension $m=1$, the equations in Lemma \ref{eqn_AHCP} have a simpler form.
\begin{lemma} [\cite{PipSin}, Lemma 2.2]\label{lemma2.2pipsin}
On a hypersurface evolving by mean curvature flow in a symmetric ambient space we have
	\begin{align*}
&\left(\partial_t-\Delta\right)|A|^2= -2|\nabla A|^2+2|A|^2\left(|A|^2+\bar{R} i c(\nu, \nu)\right) -4 \sum_{i, j, p, l}\left(h_{i j} h_{j}^p \bar{R}_{p l i} {}^l-h^{i j} h^{l p} \bar{R}_{p i l j}\right),\\
&\left(\partial_t-\Delta\right)|H|^2=-2|\nabla H|^2+2|H|^2\left(|A|^2+\bar{R} i c(\nu, \nu)\right),
	\end{align*}
where $\bar{R}ic$ is the Ricci tensor of the ambient manifold.
\end{lemma}
Another type of frames, associated more with the geometry of $\mathbb{C} P^k$, is important for computations of the components of the Riemann curvature tensor of the ambient manifold. The properties for this case are described in the lemma below.
\begin{lemma} [\cite{PipSin}, Lemma 3.1]\label{bigmatrices}
Let $\mathcal{M}$ be a submanifold of $\mathbb{C}P^k$ of dimension $n$ and codimension $m$. If $m \leq n$, then for every point $p \in \mathcal{M}$ there exist $\left\{e_1, \ldots, e_n\right\}$ an orthonormal basis of $T_p \mathcal{M}$ and $\left\{\nu_1, \ldots, \nu_m\right\}$ an orthonormal basis of $N_p \mathcal{M}$, such that:
	\begin{enumerate}
\item For every $r \leq \frac{m}{2}$ we have
	\begin{align}\label{basesJ}
\left\{\begin{array}{l}
J \nu_{2 r-1}=\tau_r e_{2 r-1}+\rho_r \nu_{2 r} \\
	J \nu_{2 r}=\tau_r e_{2 r}-\rho_r \nu_{2 r-1},
\end{array}\right.
	\end{align}with $\tau_r, \rho_r \in[0,1]$ and $\tau_r^2+\rho_r^2=1$.
\item If $m$ is odd, then $J \nu_m=e_m$.
\item The remaining vectors satisfy
	\begin{align*}
	J e_{m+1}=e_{m+2}, J e_{m+3}=e_{m+4}, \ldots, J e_{n-1}=e_n.
	\end{align*}	\end{enumerate}
\end{lemma}
Any basis satisfying the properties outlined in Lemma \ref{bigmatrices} will be called of kind (B2). Since $J^2=-i d$, from \eqref{basesJ} it follows that this kind of basis satisfies
	\begin{align}\label{basesJ2}
\left\{\begin{array}{l}
J e_{2 r-1}=-\rho_r e_{2 r}-\tau_r \nu_{2 r-1}, \\
	J e_{2 r}=\rho_r e_{2 r-1}-\tau_r \nu_{2 r}.
\end{array}\right.
	\end{align}
In the case where $m$ is odd, it is convenient to define $\tau_r=1, \rho_r=0$ for $r=\frac{m+1}{2}$. This ensures that the first equations in both \eqref{basesJ} and in \eqref{basesJ2} hold also for this value of $r$.\\
Generally, the requirements for (B1) and (B2) bases are incompatible and the two types of bases are distinct. Therefore, when using type (B2) frames, $H=\sum_\alpha H^\alpha \nu_\alpha$, with $H^\alpha$ not necessarily zero for $\alpha>1$.\\
It's worth mentioning that when $m=1$, these constructions become trivial. In this scenario, there is a unique (up to sign) normal unit vector $e_{2 k}, H$ is a multiple of this vector, and $e_1=$ $J_{e_{2 k}}$ is a tangent vector. For a hypersurface, we can choose a basis that simultaneously belongs to both (B1) and (B2) types.\\
When $m \geq 2$, from \cite{AnBa10}, we have
	\begin{align}\label{R1R2}
R_1=\sum_{\alpha, \beta}\left(\sum_{i, j} A_{i j}^\alpha A_{i j}^\beta\right)^2+\sum_{i, j, \alpha, \beta}\left(\sum_p (A_{i p}^\alpha A_{j p}^\beta-A_{i p}^\beta A_{j p}^\alpha)\right)^2 \ \text{and} \ R_2=\sum_{i, j}\left(\sum_\alpha H^\alpha A_{i j}^\alpha\right)^2.
	\end{align}
For a frame of type (B1), we easily get
	\begin{align*}
R_2= \begin{cases}|\mathring{A}|^2|H|^2+\frac{1}{n}|H|^4,&\text{ if } H \neq 0, \\ 0,&\text{ if } H=0.\end{cases}
	\end{align*}
The following result, proven in Section 3 in \cite{AnBa10} and in Section 5.2 in \cite{BakerThesis}, is valuable for estimating the reaction terms in the evolution equations of Lemma \ref{eqn_AHCP}. This result relies solely on the algebraic properties of $R_1$ and $R_2$ and is independent of the flow.
\begin{lemma}\label{lemma3.2} At a point where $H \neq 0$ we have, for any $c \in \mathbb{R}$
	\begin{align*}
	2 R_1-2 c R_2&\le2|\mathring{h}|^4-2\left(c-\frac{2}{n}\right)|\mathring{h}|^2|H|^2-\frac{2}{n}\left(c-\frac{1}{n}\right)|H|^4 +8|\mathring{h}|^2|\mathring{A}^-|^2+3|\mathring{A}^-|^4.
	\end{align*}In addition, if $\frac{1}{n}<c\le\frac{4}{3n}$ and if $d \in \mathbb{R}$ is such that $|A|^2=c|H|^2+d$, we have
	\begin{align*}
	2 R_1-2 c R_2&\leq\left(6-\frac{2}{n c-1}\right)|\mathring{h}|^2|\mathring{A}^-|^2-3|\mathring{A}^-|^4 +\frac{2 n cd}{n c-1}|\mathring{h}|^2+\frac{4 d}{n c-1}|\mathring{A}^-|^2-\frac{2 d^2}{n c-1}.
	\end{align*}
\end{lemma}
\begin{lemma}[\cite{Liu}, Lemma 3.1]\label{lemma3.3pipsin}
 For any $\eta>0$ we have the following inequalities
	\begin{align}\label{eq3.8pipsin}
	|\nabla^\perp A|^2 \geq\left(\frac{3}{n+2}-\eta\right)&|\nabla^\perp H|^2-\frac{2}{n+2}\left(\frac{2}{n+2} \eta^{-1}-\frac{n}{n-1}\right)|w|^2
	\end{align}
Here $w=\sum_{i, j, \alpha} \bar{R}_{\alpha j i j} e_i \otimes \omega_\alpha$.
	\end{lemma}
Note that if the ambient space is Einstein, like $\mathbb{C}P^k$, and if $\mathcal{M}$ is a hypersurface, then $\omega=0$. For $\eta \rightarrow 0$ in inequality \eqref{eq3.8pipsin}, we get
	\begin{align}\label{katoinequality}
|\nabla A|^2 \geq \frac{3}{n+2}|\nabla H|^2 .
	\end{align}
For $m\ge2$, $\omega$ is nonzero. Using the special properties of $\mathbb{C}P^k$, we have the following estimate.
\begin{proposition}[\cite{Naff}, Proposition 2.2]\label{prop2.2naff}
	\begin{align}\label{2.22naff}
&|\nabla^{\perp} A|^2=\sum_{i,j,k}|\hat{\nabla}_i^{\perp} A^-_{j k}+h_{j k} \nabla_i^{\perp} \nu_1|^2+\sum_{i,j,k}|\langle\nabla_i^{\perp} A^-_{j k}, \nu_1\rangle+\nabla_i h_{j k}|^2,\\
	\label{2.23naff}	&|\nabla^{\perp} H|^2  =|H|^2|\nabla^{\perp} \nu_1|^2+|\nabla| H||^2,\\
	\label{2.24naff}	&|\nabla^{\perp} A^-|^2 =|\hat{\nabla}^{\perp} A^-|^2+|\langle\nabla^{\perp} A^-, \nu_1\rangle|^2.
	\end{align}
\end{proposition}
For the orthogonal decomposition of $N \mathcal{M}$, we write $N \mathcal{M}=E_1 \oplus \hat{E}$, where $\hat{E}$ consists of normal vectors $\nu$, which are everywhere orthogonal to $\nu_1,\left\langle\nu, \nu_1\right\rangle=0$, and $E_1=C^{\infty}(M) \nu_1$. 
It is very useful to consider the implications of the Codazzi equation for the decomposition of $\nabla_i^{\perp} A_{j k}$ above. 
Following Naff in [\cite{Naff}, section 4.2], we have the following decomposition. Recalling that $A^-_{j k}$ is traceless, it is straightforward to verify that
	\begin{align}\label{eq4.15}
	&\sum_{i,j,k}|\nabla_i h_{j k}+\langle\nabla_i^{\perp} A^-_{j k}, \nu_1\rangle|^2 =\sum_{i,j,k}|\nabla_i \mathring{h}_{j k}+\langle\nabla_i^{\perp} A^-_{j k}, \nu_1\rangle|^2+\frac{1}{n}|\nabla| H \|^2,\\
\label{eq4.16} &\sum_{i,j,k}|\hat{\nabla}_i^{\perp} A^-_{j k}+h_{j k} \nabla_i^{\perp} \nu_1|^2 =\sum_{i,j,k}|\hat{\nabla}_i^{\perp} A^-_{j k}+\mathring{h}_{j k} \nabla_i^{\perp} \nu_1 |^2+\frac{1}{n}|H|^2|\nabla^{\perp} \nu_1|^2,\\
\label{eq4.17}
	&\sum_{i,j,k}|\langle\nabla_i^{\perp} \mathring{A}_{j k}, \nu_1\rangle|^2=\sum_{i,j,k}|\nabla_i \mathring{h}_{j k}+\langle\nabla_i^{\perp} A^-_{j k}, \nu_1\rangle|^2,
	\end{align}
which will be useful later on. 
From the previous result, we obtain an estimate similar to Lemma \ref{katoinequality} for general codimension, as shown in the next lemma.
\begin{lemma}[\cite{PipSin}, Lemma 3.5]\label{katoinequality2}
For any submanifold $\mathcal{M}$ of $\mathbb{C}P^k$ with dimension satisfying the assumptions of Theorem \ref{maintheorem}, we have
	\begin{align*}
	|\nabla A|^2 \geq \frac{16}{9(n+2)}|\nabla H|^2.
	\end{align*}\end{lemma}
Following Naff in [\cite{Naff}, section 4.2] and as in \eqref{eq4.15}, \eqref{eq4.16}, \eqref{eq4.17} and Lemma \ref{katoinequality2}, we obtain
	\begin{align}
&\frac{16}{9(n+2)}|\nabla| H||^2\leq\sum_{i,j,k}|\nabla_i h_{j k}+\langle\nabla_i^{\perp} A^-_{j k}, \nu_1\rangle|^2,\nonumber\\
\label{4.20naff9}
&\frac{16}{9(n+2)}|H|^2|\nabla^\bot \nu_1|^2\le\sum_{i,j,k}|\hat{\nabla}^\bot_i A^-_{jk}+h_{jk}\nabla^\bot_i \nu_1|^2
	\end{align}
and
	\begin{align}\label{eq4.21naff}
	&\frac{7n-18}{9n(n+2)}|\nabla|H||^2\le\sum_{i,j,k}|\langle\nabla^\bot_i \mathring{A}_{jk},\nu_1\rangle|^2,\\
\label{eq4.22naff}
	&\frac{7n-18}{9n(n+2)}|H|^2|\nabla^\bot\nu_1|^2\le\sum_{i,j,k}|\hat{\nabla}^\bot_i A^-_{jk}+\mathring{h}_{jk}\nabla^\bot_i\nu_1|^2.
	\end{align}
We follow \cite{HuSi09} for the definition of a parabolic neighbourhood. Given $p \in \mathcal{M}$ and $r>0$, we let $\mathcal{B}_{g(t)}(p, r) \subset \mathcal{M}$ be the closed ball of radius $r$ around $p$ with respect to the metric $g(t)$. In addition, if $t, \theta$ are such that $0 \leq t-\theta<t \leq T$, we set
	\begin{align}\label{parneigh}
&\mathcal{P}(p, t, r, \theta)=\left\{(q, s): q \in \mathcal{B}_{g(t)}(p, r), s \in[t-\theta, t]\right\},
	\end{align}
Such a set will be called a (backward) parabolic neighbourhood of $(p, t)$.

\section{Quadratic pinching condition in $\mathbb{C}P^k$}
In this section, we show the quadratic pinching condition \eqref{pinching_conditioncpk} in $\mathbb{C}P^k$ is preserved along the mean curvature flow, for a suitable positive constant. We define a new connection for the orthogonal decomposition of $N \mathcal{M}=E_1 \oplus \hat{E}$ where $\hat{E}$ consists of normal vectors $\nu$ which are everywhere orthogonal to $\nu_1,\left\langle\hat{\nu}, \nu_1\right\rangle=0$, and $E_1=C^{\infty}(\mathcal{M}) \nu_1$. Define $\hat{\nabla}^\bot$ on $\hat{E}$ by
	\begin{align*}
\hat{\nabla}_i^{\perp} \hat{\nu}:=\nabla_i^{\perp} \hat{\nu}-\left\langle\nabla_i^{\perp} \hat{\nu}, \nu_1\right\rangle \nu_1.
	\end{align*}
In order to prove that the pinching condition is preserved under the flow, we primarily follow Pipoli and Sinestrari in \cite{PipSin}. We seperate the cases of codimension one and higher codimension. Also, we prove crucial gradient estimates for the mean curvature flow.
\subsection{Preservation of the pinching}
\begin{proposition} \label{preservationhypersurface}
Let $\tilde{k}\ge 1$. Let $\mathcal{M}_0$ be a closed hypersurface of $\mathbb{C}P^k$, with $k \geq \tilde{k}+2$, of dimension $n\ge 9$ and $n\ge \frac{18+16\tilde{k}}{7}$. Then, the pinching condition
	\begin{align*}
|A|^2 \leq \frac{1}{n-\tilde{k}}|H|^2+2\tilde{k}
	\end{align*}
is preserved by the mean curvature flow for every $t\in[0,T)$.
\end{proposition}
\begin{proof}
Let us set $Z=|A|^2-\tilde{c}|H|^2-d$ with $\tilde{c}=\frac{1}{n-\tilde{k}}$ and $d=2\tilde{k}$. Lemma \ref{lemma2.2pipsin} gives
	\begin{align}\label{eq3.16pipsin}
&\left(\partial_t-\Delta\right) Z= -2\left(|\nabla A|^2-\tilde{c}|\nabla H|^2\right)+2\left(|A|^2-\tilde{c}|H|^2\right)\left(|A|^2+\bar{r}\right)  -4\left(h_{i j} h_j^p \bar{R}_{p l i}{ }^l-h^{i j} h^{l p} \bar{R}_{p i l j}\right) \nonumber\\
&= -2\left(|\nabla A|^2-\tilde{c}|\nabla H|^2\right)+2 Z\left(|A|^2+\bar{r}\right)+2 d\left(|A|^2+\bar{r}\right) -4\left(h_{i j} h_j{ }^p \bar{R}_{p l i}{ }^l-h^{i j} h^{l p} \bar{R}_{p i l j}\right),
	\end{align}
where $\bar{r}=\bar{R} i c(\nu, \nu)=2(k+1).$ By Lemma \ref{katoinequality2}, the gradient terms in \eqref{eq3.16pipsin} are non-positive. Indeed, for $\tilde{k}\ge 1$ we need at least $n\ge \frac{18+16\tilde{k}}{7}$ and $n\ge9$. The inequality $n\ge \frac{18+16\tilde{k}}{7}$ is implied by the assumptions of the theorem. Also, this inequality provides a condition on how large $\tilde{k}$ can get. So, it suffices to consider the reaction terms. Fix an orthonormal basis tangent to $\mathcal{M}_t$, which diagonalizes the second fundamental form and denote $\lambda_1 \leq \lambda_2 \leq \ldots \leq \lambda_n$ its eigenvalues. Recalling that the sectional curvature $\bar{K}_{i j}$, for any $i,j$ satisfies $\bar{K}_{i j} \geq 1$, we have
	\begin{align*}
-4\left(h_{i j} h_j^p \bar{R}_{p l i}{ }^l-h^{i j} h^{l p} \bar{R}_{p i l j}\right) & =-4\left(\lambda_j^2 \delta_{i j} \delta_{j p} \bar{R}_{p l i l}-\lambda_j \lambda_l \delta_{i j} \delta_{l p} \bar{R}_{p i l j}\right)  =-4 \sum_{j, l}\left(\lambda_j^2-\lambda_j \lambda_l\right) \bar{R}_{j l j l} \\
& =-2 \sum_{j, l}\left(\lambda_j-\lambda_l\right)^2 \bar{K}_{j l}  \leq-2 \sum_{j, l}\left(\lambda_j-\lambda_l\right)^2=-4 n|\mathring{A}|^2.
	\end{align*}
Recall from the codimension $1$ case, we have $1=2k-n$. Since $2 / \tilde{c} \geq 2 n-2\tilde{k} \ge2k+2=\bar{r}$, we have
	\begin{align*}
2 d\left(|A|^2+\bar{r}\right)-4 n\left(|A|^2-\frac{1}{n}|H|^2\right)=-\frac{4}{\tilde{c}}\left(|A|^2-\tilde{c}|H|^2-\frac{\tilde{c}}{2} d \bar{r}\right) \leq-\frac{4}{\tilde{c}}Z.
	\end{align*}
By the maximum principle, we get the desired result.
\end{proof}
In the case of higher codimension, the preservation of the quadratic pinching is derived from the following theorem.
\begin{theorem}\label{ThCPn}
Let $\tilde{k}\ge1$. Let $\mathcal{M}_0$ be a closed submanifold of $\mathbb{C}P^k$ of dimension $n\ge12$, $n\ge4\tilde{k}$ and codimension $2 \leq m<\frac{n-3}{4}$. Then, the pinching condition
	\begin{align}\label{preservedcondition}
	|A|^2 \leq \frac{1}{n-\tilde{k}}|H|^2+\frac{\tilde{k}}{n}(n-3-4m)
	\end{align}
is preserved by the flow for every $t\in[0,T)$.
\end{theorem}
\begin{proof}
Let us set $\textsl{g}=|A|^2-c_n|H|^2-d_n$, where $c_n=\frac{1}{n-\tilde{k}}$ and $d_n=\frac{\tilde{k}}{n}(n-3-4m).$ More precisely, by \eqref{R1R2} and Lemma \ref{eqn_AHCP},  we have
	\begin{align}\label{eqn_g}
	\partial_t \textsl{g}=\Delta \textsl{g}-2\left(|\nabla A|^2-c_n|\nabla H|^2\right)+2R_1-2c_n R_2+P,
	\end{align}	
where
	\begin{align}\label{7eqn_P_a}
	P=I+II+III,
	\end{align}with
	\begin{align*}
	&I=4\sum_{i,j,p,q}\bar{R}_{ipjq}\left(\sum_{\alpha} A^\alpha_{pq}A^\alpha_{ij}\right)-4\sum_{j,k,p}\bar{R}_{kjkp}\left(\sum_{i,\alpha} A^\alpha_{pi}A^\alpha_{ij}\right),\\
	&II=2\sum_{k,\alpha,\beta}\bar{R}_{k\alpha k\beta}\left(\sum_{i,j} A^\alpha_{ij}A_{ij}^\beta \right)-2c_n\sum_{k,\alpha,\beta}\bar{R}_{k\alpha k\beta} H^\alpha H^\beta,\\
	&III=-8\sum_{j,p,\alpha,\beta}\bar{R}_{jp\alpha\beta}\left(\sum_iA^\alpha_{ip}A_{ij}^\beta \right).
	\end{align*}
By Lemma \ref{katoinequality2} the gradient terms in equation \eqref{eqn_g} are non-positive, for $n\ge9$ and $n\ge \frac{18+16\tilde{k}}{7}$. 
So, it suffices to consider the reaction terms. We will treat the cases $H=0$ and $H \neq 0$. Consider a point where $\textsl{g}=0$ and $H \neq 0$. To estimate $I$, we first fix $\alpha$ and then we choose a tangent basis $\left\{\widetilde{e}_1, \ldots, \widetilde{e}_n\right\}$, not necessarily of kind (B1) or (B2), that diagonalises the entries of $A$, i.e. $A_{i j}^\alpha=\lambda_i^\alpha \delta_{i j}$.
	\begin{align*}
	I&=4\sum_{i,p}\bar{R}_{ipip}\left(\lambda^\alpha_i\lambda^\alpha_p-(\lambda^\alpha_i)^2\right)=-2\sum_{i,p}\bar{K}_{ip}\left(\lambda^\alpha_i-\lambda^\alpha_p\right)^2\le-4n|\mathring{A}^\alpha|^2.
	\end{align*}Hence, we get
	\begin{align}\label{7eqn_I}
	I\le-4n(|\mathring{h}|^2+|\mathring{A}^-|^2).
	\end{align}
We will use a basis of type (B2) for estimating the terms II and III. In order to study the term II, with our choice of basis, from \eqref{RinCPn}, we have that $\bar{R}_{s\alpha s\beta}=0$ for any $s$ if $\alpha \neq \beta$. Otherwise, we have $\bar{R}_{s \alpha s \alpha}=\bar{K}_{s \alpha}=1+3 g_{F S}\left(e_s, J \nu_\alpha\right)^2,$ which implies that $1 \leq \bar{K}_{s\alpha} \leq 1+3 \delta_{s, \alpha-n}$. Therefore, since $c_n >\frac{1}{n}$, we have
	\begin{align}\label{7eqn_II}
	II&=2 \sum_{s, \alpha} \bar{K}_{s \alpha}\left(\left|A^\alpha\right|^2-c_n\left|H^\alpha\right|^2\right)=2 \sum_{s, \alpha} \bar{K}_{s \alpha}\left(|\mathring{A}^\alpha|^2-\left(c_n-\frac{1}{n}\right)\left|H^\alpha\right|^2\right)\nonumber \\
	& \ \ \quad\quad\quad\quad\quad\quad\quad\quad\quad\quad\quad\quad\quad<2 \sum_{s, \alpha}\left(1+3 \delta_{s, \alpha-n}\right)|\mathring{A}^\alpha|^2=2(n+3)|\mathring{A}|^2.
	\end{align}
For the term III, since $\bar{R}_{j p \alpha \beta} $ is anti-symmetric in $j, p$, while $A_{j p}^\alpha$ is symmetric, we have
	\begin{align*}
	I I I=-8 \sum_{j, p, \alpha, \beta} \bar{R}_{j p \alpha \beta}\left(\sum_i A_{i p}^\alpha A_{i j}^\beta\right)=-8 \sum_{j, p, \alpha, \beta} \bar{R}_{j p \alpha \beta}\left(\sum_i \mathring{A}_{i p}^\alpha \mathring{A}_{i j}^\beta\right).
	\end{align*}
We examine the potential values of of $\bar{R}_{j p \alpha \beta}$. First fix $\alpha$ and $\beta$ coupled by \eqref{basesJ}, meaning that $\min \{\alpha, \beta\}=2 r-1$ and $\max \{\alpha, \beta\}=2 r$ for some $r \leq m / 2$. By symmetry, it suffices to consider the case where $\alpha<\beta$. We find
	\begin{align*}
	\bar{R}_{j p \alpha \beta}=\tau_r^2\left(\delta_{j, 2 r-1} \delta_{p, 2 r}-\delta_{j, 2 r} \delta_{p, 2 r-1}\right)-2 \rho_r g_{F S}\left(e_j, J e_p\right),
	\end{align*}and
	\begin{align*}
g_{F S}\left(e_j, J e_p\right)=\left\{\begin{array}{llll}
-\rho_s&\text{ if } j=2 s,&p=2 s-1,&s \leq \frac{m}{2}, \\
	\rho_s&\text{ if } j=2 s-1,&p=2 s,&s \leq \frac{m}{2}, \\
	1&\text{ if } j=m+2 s,&p=m+2 s-1,&s \leq \frac{n-m}{2}, \\
	-1&\text{ if } j=m+2 s-1,&p=m+2 s,&s \leq \frac{n-m}{2}, \\
	0&\text{ otherwise. }&&
\end{array}\right.
	\end{align*}If $\alpha$ and $\beta$ aren't coupled by the system \eqref{basesJ}, there are two indices $r \neq s$, such that $\alpha$ is (or is coupled with) $\nu_{2 r-1}$ and $\beta$ is (or is coupled with) $\nu_{2 s-1}$. In this situation, we have
	\begin{align*}
	\bar{R}_{jp\alpha \beta}=\tau_r \tau_s\left(\delta_{j, \alpha-n} \delta_{p, \beta-n}-\delta_{j, \beta-n} \delta_{p, \alpha-n}\right).
	\end{align*}Using this and by summing all similar terms, we arrive at
	\begin{align*}
	&I I I=16 \sum_{r\le\frac{m}{2}}\left(2 \rho_r^2-\tau_r^2\right) \sum_i\left(\mathring{A}_{i\ 2 r}^{2 r-1} \mathring{A}_{i\ 2 r-1}^{2 r}-\mathring{A}_{i\ 2 r-1}^{2 r-1} \mathring{A}_{i\ 2 r}^{2 r}\right) \\
	&-8 \sum_{r\neq s<\frac{m}{2}} \tau_r \tau_s \sum_i\left(\mathring{A}_{i\ 2 s}^{2 r} \mathring{A}_{i\ 2 r}^{2 s}-\mathring{A}_{i\ 2 r}^{2 r} \mathring{A}_{i\ 2 s}^{2 s}\right) \\
	&-16 \sum_{r \neq s, r \leq \frac{m}{2}, s \leq \frac{m+1}{2}} \tau_r \tau_s \sum_i\left(\mathring{A}_{i\ 2 s-1}^{2 r} \mathring{A}_{i\ 2 r}^{2 s-1}-\mathring{A}_{i\ 2 r}^{2 r} \mathring{A}_{i\ 2s-1}^{2 s-1}\right) \\
	&-8 \sum_{r \neq s \leq \frac{m+1}{2}} \tau_r \tau_s \sum_i\left(\mathring{A}_{i\ 2 s-1}^{2 r-1} \mathring{A}_{i\ 2r-1}^{2 s-1}-\mathring{A}_{i\ 2 r-1}^{2 r-1} \mathring{A}_{i\ 2 s-1}^{2 s-1}\right) \\
	&+32 \sum_{r\neq s\leq \frac{m}{2}} \rho_r \rho_s \sum_i\left(\mathring{A}_{i\ 2 s}^{2 r-1} \mathring{A}_{i\ 2s-1}^{2 r}-\mathring{A}_{i\ 2 s-1}^{2r-1} \mathring{A}_{i\ 2 s}^{2 r}\right) \\
	&+32 \sum_{r \leq \frac{m}{2}} \rho_r \sum_{s \leq \frac{n-m}{2}} \sum_i\left(\mathring{A}_{i\ m+2 s-1}^{2 r-1} \mathring{A}_{i\ m+2 s}^{2 r}-\mathring{A}_{i\ m+2 s}^{2 r-1} \mathring{A}_{i\ m+2 s-1}^{2 r}\right). \\
	\end{align*}
Note that $I I I \leq|I I I|$. Using the triangle inequality and Young's inequality repeatedly and considering that for any $r$ and $s$ that $\left|2 \rho_r^2-\tau_r^2\right| \leq 2, \ \left|\tau_r \tau_s\right| \leq 1, \ \left|\rho_r \rho_s\right| \leq 1, \ \left|\rho_r\right| \leq 1,$ we get
	\begin{align*}
	&I I I \leq 16 \sum_{r \leq \frac{m}{2}} \sum_i\left(\left|\mathring{A}_{i\ 2 r}^{2 r-1}\right|^2+\left|\mathring{A}_{i\ 2 r-1}^{2 r}\right|^2+\left|\mathring{A}_{i\ 2 r-1}^{2 r-1}\right|^2+\left|\mathring{A}_{i\ 2 r}^{2 r}\right|^2\right) \\
	&+4 \sum_{r \neq s \leq \frac{m}{2}} \sum_i\left(\left|\mathring{A}_{i\ 2 s}^{2 r}\right|^2+\left|\mathring{A}_{i\ 2 r}^{2 s}\right|^2+\left|\mathring{A}_{i\ 2 r}^{2 r}\right|^2+\left|\mathring{A}_{i\ 2 s}^{2 s}\right|^2\right) \\
	&+8 \sum_{r \neq s, r \leq \frac{m}{2}, s \leq \frac{m+1}{2}} \sum_i\left(\left|\mathring{A}_{i\ 2 s-1}^{2 r}\right|^2+\left|\mathring{A}_{i\ 2 r}^{2 s-1}\right|^2+\left|\mathring{A}_{i\ 2 r}^{2 r}\right|^2+\left|\mathring{A}_{i\ 2 s-1}^{2 s-1}\right|^2\right) \\
	&+4 \sum_{r \neq s \leq \frac{m+1}{2}} \sum_i\left(\left|\mathring{A}_{i\ 2 s-1}^{2 r-1}\right|^2+\left|\mathring{A}_{i\ 2r-1}^{2 s-1}\right|^2+\left|\mathring{A}_{i\ 2 r-1}^{2 r-1}\right|^2+\left|\mathring{A}_{i\ 2 s-1}^{2 s-1}\right|^2\right) \\
	&+16\sum_{r\neq s\le\frac{m}{2}}\sum_i \left(\left|\mathring{A}^{2r-1}_{i\ 2s}\right|^2+\left|\mathring{A}^{2r}_{i\ 2s-1}\right|^2+\left|\mathring{A}^{2r-1}_{i\ 2s-1}\right|^2+\left|\mathring{A}^{2r}_{i\ 2s}\right|^2\right)\\
	&+16 \sum_{r \leq \frac{m}{2}, s \leq \frac{n-m}{2}} \sum_i\left(\left|\mathring{A}_{i\ m+2 s-1}^{2 r-1}\right|^2+\left|\mathring{A}_{i\ m+2 s}^{2 r}\right|^2+\left|\mathring{A}_{i\ m+2 s}^{2 r-1}\right|^2+\left|\mathring{A}_{i\ m+2 s-1}^{2 r}\right|^2\right).
	\end{align*}
Note that, if codimension $m=2$, there are no indices $r \neq s \leq \frac{m+1}{2}$. Then, some of the sums in the expressions above are empty and so
	\begin{align*}
	I I I \leq 16|\mathring{A}|^2.
	\end{align*}If $m>2$, by collecting similar terms we find
	\begin{align*}
&I I I \leq \sum_{r\le\frac{m}{2}}\sum_i\left(8m\left|\mathring{A}_{i\ 2 r}^{2 r-1}\right|^2+8m\left|\mathring{A}_{i\ 2 r-1}^{2 r}\right|^2+8 m\left|\mathring{A}_{i\ 2 r}^{2 r}\right|^2+8 m\left|\mathring{A}_{i\ 2 r-1}^{2 r-1}\right|^2\right) \\
	&+24 \sum_{r \neq s \leq \frac{m}{2}}\sum_i\left(\left|\mathring{A}_{i\ 2 s}^{2 r}\right|^2+\left|\mathring{A}_{i\ 2 s}^{2 r-1}\right|^2+\left|\mathring{A}_{i\ 2 s-1}^{2 r}\right|^2+\left|\mathring{A}_{i\ 2 s-1}^{2 r-1}\right|^2\right) \\
	&+16 \sum_{r, s \leq \frac{n-m}{2}}\sum_i\left(\left|\mathring{A}_{i\ m+2 s-1}^{2 r-1}\right|^2+\left|\mathring{A}_{i\ m+2 s}^{2 r}\right|^2+\left|\mathring{A}_{i\ m+2 s}^{2 r-1}\right|^2+\left|\mathring{A}_{i\ m+2 s-1}^{2 r}\right|^2\right) \\
	&\leq 8 m|\mathring{A}|^2.
	\end{align*}So, we can say in any case
	\begin{align}\label{7eqn_III}
	I I I \leq 8 m|\mathring{A}|^2.
	\end{align}
In conclusion, by \eqref{7eqn_I}, \eqref{7eqn_II} and \eqref{7eqn_III} we get
	\begin{align}\label{P_aestimate}
	P=I+I I+I I I \leq-2(n-3-4 m)|\mathring{A}|^2.
	\end{align}
Note here that $n-3-4 m$ is positive, for any $ 2\le m<\frac{n-3}{4}$ and so we require at least $n\ge12$. Let $R=2 R_1-2 c_n R_2+P$. Considering a frame of type (B1), Lemma \eqref{lemma3.2} says that at any point with $g=0$, we get
	\begin{align*}
	R&\le\left(6-\frac{2}{n c_n-1}\right)|\mathring{A}|^2 |\mathring{A}^-|^2+\left(\frac{2 n c_nd_n }{n c_n-1}-2(n-3-4 m)\right) |\mathring{h}|^2-3|\mathring{A}^-|^4 \\
	&+\left(\frac{4 d_n}{n c_n-1}-2(n-3-4 m)\right)|\mathring{A}^-|^2-\frac{2 d_n^2}{n c_n-1}.
	\end{align*}
Note that, for our choice of $c_n$ and $d_n$ the coefficient of the term $|\mathring{A}|^2|\mathring{A}^-|^2$ is nonpositive, for any $n\ge4\tilde{k}$, while the coefficient of the term $|\mathring{h}|^2$ is zero. Moreover, the assumptions $\textsl{g}=0$ and $c_n>1/n$, imply that $|\mathring{A}|^2 \geq d_n$. Using this, we have
	\begin{align*}
	R &\leq-3 |\mathring{A}^-|^4+\left(\left(6-\frac{2}{n c_n-1}\right) d_n+\frac{4 d_n}{n c_n-1}-2(n-3-4 m)\right) |\mathring{A}^-|^2-\frac{2 d_n^2}{n c_n-1} \\
	&=-3| \mathring{A}^-|^4+4 d_n |\mathring{A}^-|^2+2 d_n(d_n-n+3+4 m).
	\end{align*}Using $4 d_n|\mathring{A}^-|^2 \leq 3|\mathring{A}^-|^4+\frac{4}{3} d_n^2$, we conclude that $R \leq 2 d_n\left(\frac{5}{3} d_n-n+3+4 m\right).$ Our choice of $d_n$, implies $R<0$, for any $n\ge \frac{5}{3}\tilde{k}$. Now, consider the case of a point where $\textsl{g}=0$ and $H=0$. In this case, we have $|A|^2=$ $|\mathring{A}|^2=d_n, R_2=0$. In addition, using Theorem 1 from \cite{Li1992}, we find $2 R_1 \leq 3|A|^4=3 d_n^2$. As before, we obtain that $P \leq-2(n-3-4 m)|\mathring{A}|^2=-2(n-3-4 m) d_n$. Therefore,
	\begin{align*}
	R \leq 3 d_n^2-2(n-3-4 m) d_n,
	\end{align*}which is negative for any $n>\frac{3}{2}\tilde{k}$. By the maximum principle, we get the desired result.
\end{proof}
\subsection{Gradient estimates}
This section presents a proof of the gradient estimates for the mean curvature flow in the case of $\mathbb{C}P^k$ as the background manifold. Just like in Chapter 4, we establish this estimate directly from the quadratic curvature bound $|A|^2 < c_n |H|^2 + d_n$, where $c_n = \frac{1}{n-\tilde{k}}$, $\tilde{k}\ge1$, without relying on the asymptotic cylindrical estimates. In fact, we demonstrate the cylindrical estimates follow as a consequence of the gradient estimates we derive here. These estimates are pointwise gradient estimates that rely solely on the mean curvature (or, equivalently, the second fundamental form) at a point and not on the maximum of curvature, as is the case with more general parabolic-type derivative estimates. Here, we make use of the Kato-type inequality from Lemma \ref{katoinequality2}.
\begin{theorem}\label{7thm_gradient}
Let $ \mathcal{M}_t , t \in [0,T)$ be a closed $n$-dimensional solution to the mean curvature flow in the Riemannian manifold $\mathbb{C}P^k$, where
	\begin{align*}
	|A|^2 - \tilde{c}|H|^2 -\tilde{d}<0, |H| >0
	\end{align*}
with $ \tilde{c}=\frac{1}{n-\tilde{k}}$, $\tilde{d}\ge0$ and $\tilde{k}\ge1$, $n\ge \max\{12, 4\tilde{k}\}$. 
 Then, there exists a constant $ \gamma_1= \gamma_1 (n, \mathcal M_0)$ and a constant $ \gamma_2 = \gamma_2 (n, \mathcal{M}_0)$, such that the flow satisfies the uniform estimate
	\begin{align*}
	|\nabla A|^2 \leq \gamma_1 |A|^4+\gamma_2,
	\end{align*}
 for every $t\in [0, T)$.
\end{theorem}
\begin{proof}
The proof is the same in the case of hypersurfaces, from Proposition \ref{preservationhypersurface}, with assumptions $n\ge\max\{9, \frac{18+16\tilde{k}}{7}\}$. We choose here $ \tilde{\kappa}_n = \left( \frac{16}{9(n+2)}-\tilde{c}\right)>0$. We will consider here the evolution equation for $\frac{|\nabla A|^2}{g^2},$ where $ g = \tilde{c}|H|^2-|A|^2+\tilde{d}>0$. Since $ |A|^2-\tilde{c}|H|^2 -\tilde{d}< 0, |H|>0$ and $\mathcal{M}_0$ is compact, there exists an $ \eta(\mathcal{M}_0) >0, C_\eta(\mathcal{M}_0)>0$, so that 
	\begin{align*}
	\left( \tilde{c}-\eta \right)|H|^2-|A|^2+\tilde{d} \geq C_\eta>0.
	\end{align*}
Hence, 
 	\begin{align*}
	g=\tilde{c}|H|^2-|A|^2+\tilde{d}\ge\eta|H|^2>\frac{\eta}{\tilde{c}}|A|^2>\e_1|A|^2+\e_2,
	\end{align*}
where $ \e_1< \frac{\eta}{\tilde{c}}$ and $\e_2>0$. From  Lemma \ref{katoinequality2} and Theorem \ref{ThCPn}, we get
	\begin{align*}
	\left(\partial_t-\Delta\right) g&=-2 \left( \tilde{c}|\nabla H|^2-|\nabla A|^2 \right)+2 \left( \tilde{c} R_2-R_1 \right)-P\geq-2 \left(\left(\frac{16}{9(n+2)}\right)^{-1}\tilde{c}-1\right) |\nabla A|^2\\
	&\ge-2\left(\frac{9(n+2)}{16}\tilde{c}-1\right)|\nabla A|^2= 2\tilde{\kappa}_n \frac{9(n+2)}{16}| \nabla A|^2.
	\end{align*}
From [\cite{Liu}, section 4], the evolution equation for $ |\nabla A|^2 $ is given by 
 	\begin{align*}
	\left(\partial_t-\Delta\right) |\nabla A|^2&\le-2 |\nabla^2 A|^2+\tilde{c} |A|^2 |\nabla A|^2+\tilde{d}|\nabla A|^2+c_1|A||\nabla A|+c_2|\nabla A|.
	\end{align*}
Let $w,z$ satisfy the evolution equations $\partial_tw = \Delta w+W$ and $\partial_tz = \Delta z+Z$, then we find
 	\begin{align*}
	\left(\partial_t-\Delta\right)\frac{w}{z}&=\frac{2}{z}\left\la \nabla \left( \frac{w}{z}\right) , \nabla z \right\ra+\frac{W}{z}-\frac{w}{z^2} Z=2\frac{\la \nabla w , \nabla z \ra}{z^2}-2 \frac{w|\nabla z |^2}{z^3}+\frac{W}{z}-\frac{w}{z^2} Z.
	\end{align*}
Furthermore, for any function $g$, we have by Kato's inequality
	\begin{align*}
	\la \nabla g , \nabla |\nabla A|^2 \ra&\leq 2 |\nabla g| |\nabla^2 A| |\nabla A| \leq \frac{1}{g}|\nabla g |^2 | \nabla A|^2+g |\nabla^2 A|^2.
	\end{align*}
We then get
	\begin{align*}
	-\frac{2}{g}| \nabla^2 A|^2+\frac{2}{g}\left\la \nabla g ,\nabla \left( \frac{|\nabla A|^2}{g}\right) \right\ra =-\frac{2}{g}| \nabla^2 A|^2-\frac{2}{g^3}|\nabla g|^2 |\nabla A|^2+\frac{2}{g^2}\la \nabla g ,\nabla |\nabla A|^2 \ra \leq 0.
	\end{align*}
Then, if we let $ w = |\nabla A|^2 $ and $ z = g$, with $W \le-2 |\nabla^2 A|^2+\tilde{c} |A|^2 |\nabla A|^2+\tilde{d}|\nabla A|^2+c_1|A||\nabla A|+c_2|\nabla A|$ and $Z\geq 2\tilde{\kappa}_n \frac{9(n+2)}{16}| \nabla A|^2 $, we get
	\begin{align*}
	\left(\partial_t-\Delta\right) \frac{|\nabla A|^2}{g}&\leq \frac{2}{g}\left\la \nabla g ,\nabla \left( \frac{|\nabla A|^2}{g}\right) \right\ra+\frac{1}{g}(-2 |\nabla^2 A|^2+\tilde{c} |A|^2 |\nabla A|^2 +\tilde{d}|\nabla A|^2\\
	&+c_1|A||\nabla A|+c_2|\nabla A|)-2 \tilde{\kappa}_n \frac{9(n+2)}{16}\frac{|\nabla A|^4}{g^2} \\
	&\leq \tilde{c} |A|^2 \frac{|\nabla A|^2}{g}+\tilde{d}\frac{|\nabla A|^2}{g}+c_1|A|\frac{|\nabla A|}{g}+c_2\frac{|\nabla A|}{g}-2 \tilde{\kappa}_n \frac{9(n+2)}{16}\frac{|\nabla A|^4}{g^2}.
	\end{align*}
We repeat the above computation with $w = \frac{|\nabla A|^2}{g}, z = g,$
	\begin{align*}
	W\leq \tilde{c} |A|^2 \frac{|\nabla A|^2}{g}+\tilde{d}\frac{|\nabla A|^2}{g}+c_1|A|\frac{|\nabla A|}{g}+c_2\frac{|\nabla A|}{g}-2 \tilde{\kappa}_n \frac{9(n+2)}{16}\frac{|\nabla A|^4}{g^2}
	\end{align*}
$ Z\geq 0$ and assume $\frac{|\nabla A|}{g}>1$, we get
	\begin{align*}
	\left(\partial_t-\Delta\right)\frac{|\nabla A|^2}{g^2}&\leq \frac{2}{g}\left\la \nabla g ,\nabla \left( \frac{|\nabla A|^2}{g^2}\right) \right\ra+\frac{1}{g}\left( \tilde{c} |A|^2 \frac{|\nabla A|^2}{g}+\tilde{d}\frac{|\nabla A|^2}{g}-2 \tilde{\kappa}_n \frac{9(n+2)}{16}\frac{|\nabla A|^4}{g^2}\right)\\
	&+c_3|A|^2\frac{|\nabla A|^2}{g^2}+c_4\frac{|\nabla A|^2}{g^2}.
 	\end{align*}
The nonlinearity then is
	\begin{align*}
	\frac{|\nabla A|^2}{g^2} \left( (\tilde{c}+c_3)|A|^2+(\tilde{d}+c_4)-\frac{2 \tilde{\kappa}_n 9(n+2)}{16}\frac{|\nabla A|^2}{g} \right).
	\end{align*}
Since $g>\e_1|A|^2+\e_2,$ there exists a constant $N$, such that
	\begin{align*}
	Ng\ge (\tilde{c}+c_3)|A|^2+\tilde{d}+c_4.
	\end{align*}
Hence, by the maximum principle, there exists a constant (with $\eta,\e_1,\e_2$ chosen sufficiently small so that N is sufficiently large, this estimate holds at the initial time), such that
	\begin{align*}
	\frac{|\nabla A|^2}{g^2}\leq \frac{8N}{ \tilde{\kappa}_n 9(n+2)}.
	\end{align*}
Therefore, we see there exists a constant $\mathcal{C} = \frac{8N}{ \tilde{\kappa}_n 9(n+2)}= \mathcal{C}(n, \mathcal{M}_0) $, such that $\frac{|\nabla A|^2}{g^2}\leq \mathcal{C}$ and from the definition of $g$, we get the result of the lemma.
\end{proof}
Higher order estimates on $\left|\nabla^n A\right|^2$, for $n=2$, follows from Theorem 3.4, in \cite{HNAV}.  Also, for all $n$ the proof follows by an analogous method. Furthermore, we derive estimates on the time derivative of the second fundamental form since we have the evolution equation
	\begin{align*}
	\left|\partial_t A\right|=|\Delta A+A * A * A| \leq C|\nabla^2 A|+C|A|^3 \leq c_1|A|^3+c_2.
	\end{align*}
\section{Convergence for finite time in $\mathbb{C}P^k$}
%
%
%
The goal of this section is to prove case (1) of the theorem below.
\begin{theorem}
Let $\mathcal{M}_0$ be a closed submanifold of $\mathbb{C}P^k$ of real dimension $n$ and codimension $m = 2k - n$. Let $\tilde{k}=2$. For $m=1$, we have $n\ge\max \{9, 4\tilde{k}\}$ and for $2\le m<\frac{n-3}{4}$, we have $n\ge\max\{23, 4\tilde{k}\}$. If at every point of $\mathcal{M}_0$ the inequality
	\begin{align*}
|A|^2< \begin{cases}\frac{1}{n-\tilde{k}}|H|^2+2\tilde{k} & \text { if } m=1, \\ \frac{1}{n-\tilde{k}}|H|^2+\frac{\tilde{k}(n-3-4 m)}{n} & \text { if } 2\le m <  \frac{n-3}{4},\end{cases}
	\end{align*}
is satisfied, then the same holds on $\mathcal{M}_t$ for all $0 < t < T$, where $T$ is the maximal time of existence. In addition, one of the following holds.
	\begin{enumerate}
\item If $T< \infty$, then the blow up of $\mathcal{M}_t$ is a codimension one ancient flow as $t \to T$,
\item If $T = \infty$, then $\mathcal{M}_t$ converges to a smooth totally geodesic submanifold as $t \to T$.
	\end{enumerate}
Case (2) can only occur if $n$ is even, and the limit submanifold is isometric to $\mathbb{C}P^{\frac{n}{2}}$.
	\end{theorem}

\subsection{The evolution of $\frac{|A^-|^2}{f}$}
For the case $m=1$, there is no point for a codimension one argument. Following the calculations for the evolution equation of $|A^-|$ from \cite{HNAV},
we arrive at
	\begin{align*}
	&\left(\partial_t-\Delta\right)\frac{\sum_{i,j}|\langle A_{ij},H\rangle|^2}{|H|^2}=|H|^{-2}\left(\partial_t-\Delta\right)\sum_{i,j}|\langle A_{ij},H\rangle|^2-2|h|^4+2|h|^2\sum_k|\nabla^\bot_k \nu_1|^2\\
	&\quad\quad\quad+2|H|^{-2}|h|^2|\nabla|H||^2+8|H|^{-1} \sum_{i,j,k}\nabla_k |H|h_{ij}\nabla_k h_{ij}-2|h|^2|H|^{-2}\sum_{k,\alpha,\beta}\bar{R}_{k\alpha k\beta} H^\alpha H^\beta.
	\end{align*}
Recall \eqref{Riccieq}. Define a vector-valued version of the normal curvature by
	\begin{align}
	R^\bot_{ij}(\nu_\alpha)=R^\bot_{ij\alpha\beta}\nu_\beta= \bar{R}_{ij\alpha\beta}\nu_\beta+(A^\alpha_{ip}A^\beta_{jp}-A^\beta_{ip}A^\alpha_{jp})\nu_\beta.
	\end{align}In particular, we note that $R^\bot_{ij}(\nu_1)=\bar{R}_{ij}(\nu_1)+h_{ip}A^\beta_{jp}-{h}_{jp}A^\beta_{ip}$, which in view of
	\begin{align*}
	A_{ij}=A^-_{ij}+h_{ij}\nu_1=A^-_{ij}+\mathring{h}_{ij}\nu_1+\frac{1}{n}|H|g_{ij}\nu_1,
	\end{align*}gives
	\begin{align}\label{7eqn_Rbot}
	R^\bot_{ij}(\nu_1)=\bar{R}_{ij}(\nu_1)+\mathring{h}_{ip}A^-_{jp}-\mathring{h}_{jp}A^-_{ip}.
	\end{align}	
For the evolution equation of $\langle A_{ij},H\rangle$, we have the following lemma.
\begin{lemma}\label{7B}
The evolution equation of $|\langle A_{ij},H\rangle|^2$ is
	\begin{align*}
	|H|^{-2}\left(\partial_t-\Delta\right)|\langle A_{ij},H\rangle|^2&=4|\mathring{h}_{ij}A_{ij}^- |^2+2|R_{ij}^\bot (\nu_1)-\bar{R}_{ij}(\nu_1)|^2+4|h|^4-4|H|^{-1}\mathring{h}_{ij}\nabla_k |H|\langle\nabla^\bot_k A_{ij}^- ,\nu_1\rangle\\
	&-4\mathring{h}_{ij}\langle\nabla^\bot_k A_{ij}^- ,\nabla^\bot_k\nu_1\rangle-4|h|^2|\nabla^\bot_k\nu_1|^2-2|H|^{-2}|h|^2|\nabla|H||^2\\
	&-8|H|^{-1} \nabla_k |H|h_{ij}\nabla_k h_{ij}-2|\nabla h|^2+2B',
		\end{align*}where
	\begin{align*}
	B'&:=2|H|^{-2}\bar{R}_{ipjq}\langle A_{pq},H\rangle\langle A_{ij},H\rangle-2|H|^{-2}\bar{R}_{kjkp} \langle A_{pi},H\rangle\langle A_{ij},H\rangle\\
	&+|H|^{-2}A^\alpha_{ij}\bar{R}_{k\alpha k\beta}H^\beta\langle A_{ij},H\rangle+|H|^{-2}H^\alpha\bar{R}_{k\alpha k\beta} A_{ij}^\beta \langle A_{ij},H\rangle\\
	&-2|H|^{-2}A^\alpha_{jp} \bar{R}_{ip\alpha\beta} H^\beta\langle A_{ij},H\rangle -2|H|^{-2}A^\alpha_{ip}\bar{R}_{jp\alpha\beta} H^\beta \langle A_{ij},H\rangle.
	\end{align*}\end{lemma}
\begin{proof}
Whenever $h$ is traced with $A^-$ or its derivative, we may replace $h$ with $\mathring{h}$, because $A^-$ is traceless. Also, for simplicity, we avoid the summation notation. To begin with, using \eqref{7eqn_A}, we substitute formulas
	\begin{align*}
	\Big\langle\left(\partial_t-\Delta\right)^\bot A_{ij},H\Big\rangle&=\langle A_{ij},A_{pq}\rangle \langle A_{pq},H\rangle+\langle A_{iq},A_{pq}\rangle \langle A_{pj},H\rangle+\langle A_{jq},A_{pq}\rangle\langle A_{pi},H\rangle\\
	&-2\langle A_{ip},A_{jq}\rangle \langle A_{pq}, H\rangle+2\bar{R}_{ipjq}\langle A_{pq},H\rangle-\bar{R}_{kjkp} \langle A_{pi},H\rangle-\bar{R}_{kikp}\langle A_{pj},H\rangle\\
	&+A^\alpha_{ij}\bar{R}_{k\alpha k\beta} \langle\nu_\beta, H\rangle-2A^\alpha_{jp}\bar{R}_{ip\alpha\beta}\langle \nu_\beta, H\rangle-2A^\alpha_{ip}\bar{R}_{jp\alpha\beta}\langle\nu_\beta, H\rangle,\\
	\Big\langle A_{ij} ,\left(\partial_t-\Delta\right)^\bot H\Big\rangle&=\langle A_{pq}, H\rangle\langle A_{pq},A_{ij}\rangle+H^\alpha\bar{R}_{k\alpha k\beta}\langle \nu_\beta, A_{ij}\rangle.
	\end{align*}Tracing each of the equations with a copy of $\langle A_{ij},H\rangle$, we get
	\begin{align*}
	\Big\langle\left(\partial_t-\Delta\right)^\bot A_{ij},H\Big\rangle\langle A_{ij},H\rangle&=\langle A_{ij},A_{pq}\rangle \langle A_{pq},H\rangle\langle A_{ij},H\rangle+2\langle A_{iq},A_{pq}\rangle \langle A_{pj},H\rangle\langle A_{ij},H\rangle\\
	&-2\langle A_{ip},A_{jq}\rangle \langle A_{pq},H\rangle\langle A_{ij},H\rangle+2\bar{R}_{ipjq}\langle A_{pq},H\rangle\langle A_{ij},H\rangle\\
	&-2\bar{R}_{kjkp} \langle A_{pi},H\rangle\langle A_{ij},H\rangle+A^\alpha_{ij}\bar{R}_{k\alpha k\beta}\langle \nu_\beta, H\rangle\langle A_{ij},H\rangle\\
	&-2A^\alpha_{jp} \bar{R}_{ip\alpha\beta}\langle \nu_\beta, H\rangle\langle A_{ij},H\rangle -2A^\alpha_{ip}\bar{R}_{jp\alpha\beta}\langle \nu_\beta, H\rangle \langle A_{ij},H\rangle,\\
	\Big\langle A_{ij} ,\left(\partial_t-\Delta\right)^\bot H\Big\rangle\langle A_{ij},H\rangle&=\langle A_{pq}, H\rangle\langle A_{pq},A_{ij}\rangle\langle A_{ij},H\rangle+H^\alpha\bar{R}_{k\alpha k\beta}\langle\nu_\beta, A_{ij}\rangle\langle A_{ij},H\rangle.
	\end{align*}Putting the above equations together and keeping in mind that $\langle \nu_\beta,H\rangle=H^\beta$ we have,
	\begin{align*}
	\left(\left(\partial_t-\Delta\right)\langle A_{ij},H\rangle\right)\langle A_{ij},H\rangle&=2\langle A_{ij},A_{pq}\rangle \langle A_{pq},H\rangle\langle A_{ij},H\rangle+2\langle A_{iq},A_{pq}\rangle \langle A_{pj},H\rangle\langle A_{ij},H\rangle\\
	&-2\langle A_{ip},A_{jq}\rangle \langle A_{pq},H\rangle\langle A_{ij},H\rangle +2\bar{R}_{ipjq}\langle A_{pq},H\rangle\langle A_{ij},H\rangle\\
	&-2\bar{R}_{kjkp} \langle A_{pi},H\rangle\langle A_{ij},H\rangle+A^\alpha_{ij}\bar{R}_{k\alpha k\beta} H^\beta\langle A_{ij},H\rangle\\
	&-2A^\alpha_{jp} \bar{R}_{ip\alpha\beta} H^\beta\langle A_{ij},H\rangle -2A^\alpha_{ip}\bar{R}_{jp\alpha\beta} H^\beta \langle A_{ij},H\rangle\\
	&+H^\alpha\bar{R}_{k\alpha k\beta} A_{ij}^\beta \langle A_{ij},H\rangle-2\langle \nabla^\bot_k A_{ij},\nabla^\bot_k H\rangle\langle A_{ij},H\rangle.
	\end{align*}Define
	\begin{align*}
	B&:=2\bar{R}_{ipjq}\langle A_{pq},H\rangle\langle A_{ij},H\rangle-2\bar{R}_{kjkp} \langle A_{pi},H\rangle\langle A_{ij},H\rangle\\
	&+A^\alpha_{ij}\bar{R}_{k\alpha k\beta} H^\beta\langle A_{ij},H\rangle+H^\alpha\bar{R}_{k\alpha k\beta} A_{ij}^\beta \langle A_{ij},H\rangle\\
	&-2A^\alpha_{jp} \bar{R}_{ip\alpha\beta} H^\beta\langle A_{ij},H\rangle -2A^\alpha_{ip}\bar{R}_{jp\alpha\beta} H^\beta \langle A_{ij},H\rangle.
	\end{align*}We use the Uhlenbeck's trick to suppose that we are in an orthogonal frame. That is, suppose $g^{ij}=\delta_{ij}$ remains orthogonal along the flow. More precisely, for any $e_i,e_j$ orthonormal, we have $\partial_t g^{ij}=\partial_t \langle e_i,e_j\rangle=0.$ Therefore, excluding the time derivative of the inverse of the metric, which is the term
	\begin{align*}
	2\left(\partial_tg^{ij}\right)g^{pq}\langle A_{ip},H\rangle\langle A_{jq},H\rangle,
	\end{align*}we have
	\begin{align}\label{7eqn_AH}
	&\left(\partial_t-\Delta\right)|\langle A_{ij},H\rangle|^2=2\Big(\left(\partial_t-\Delta\right)\langle A_{ij},H\rangle\Big)\langle A_{ij},H\rangle-2|\nabla\langle A_{ij},H\rangle|^2\nonumber\\
	&\quad\quad=4\langle A_{ij},A_{pq}\rangle \langle A_{pq},H\rangle\langle A_{ij},H\rangle+4\langle A_{iq},A_{pq}\rangle \langle A_{pj},H\rangle\langle A_{ij},H\rangle\nonumber\\
	&\quad\quad-4\langle A_{ip},A_{jq}\rangle \langle A_{pq},H\rangle\langle A_{ij},H\rangle-4\langle\nabla^\bot_k A_{ij},\nabla_k^\bot H\rangle\langle A_{ij},H\rangle-2|\nabla\langle A_{ij},H\rangle|^2+2B.
	\end{align}
To finish the proof, we multiply $|H|^{-2}$ and then rewrite each of the remaining terms using $A=A^-+h\nu_1$. For the first term on the first line of \eqref{7eqn_AH}, we have
	\begin{align}\label{7eqn2}
	4|H|^{-2}\langle A_{ij},A_{pq}\rangle \langle A_{pq},H\rangle\langle A_{ij},H\rangle&=4|H|^{-2}|H|^2h_{ij} h_{pq}\langle A_{ij},A_{pq}\rangle\nonumber\\
	&=4|h|^4+4h_{ij} h_{pq}\langle A_{ij}^- ,A_{pq}^- \rangle\nonumber\\
	&=4|h|^4+4\mathring{h}_{ij}\mathring{h}_{pq}\langle A_{ij}^- ,A_{pq}^- \rangle\nonumber\\
	&=4|h|^4+4|\mathring{h}_{ij}A_{ij}^- |^2.
	\end{align}Also, B can be rewritten as
	\begin{align*}
	B'&:=2|H|^{-2}\bar{R}_{ipjq}\langle A_{pq},H\rangle\langle A_{ij},H\rangle-2|H|^{-2}\bar{R}_{kjkp} \langle A_{pi},H\rangle\langle A_{ij},H\rangle+|H|^{-2}A^\alpha_{ij}\bar{R}_{k\alpha k\beta}H^\beta\langle A_{ij},H\rangle\\
	&+|H|^{-2}H^\alpha\bar{R}_{k\alpha k\beta} A_{ij}^\beta \langle A_{ij},H\rangle-2|H|^{-2}A^\alpha_{jp} \bar{R}_{ip\alpha\beta} H^\beta\langle A_{ij},H\rangle -2|H|^{-2}A^\alpha_{ip}\bar{R}_{jp\alpha\beta} H^\beta \langle A_{ij},H\rangle\\
	&=|H|^{-2}B.
	\end{align*}
For the difference of second and third term of \eqref{7eqn_AH}, we notice the resemblance to $|R_{ij}^\bot (\nu_1)|^2$ in \eqref{7eqn_Rbot}. We compute
	\begin{align}\label{7eqn3}
	|\mathring{h}_{ip}A_{jp}^- -\mathring{h}_{jp}A_{ip}^- |^2&=|h_{ip}A_{jp} -h_{jp}A_{ip}|^2=\langle h_{ip}A_{jp} -h_{jp}A_{ip} ,h_{iq}A_{jq} -h_{jq}A_{iq}\rangle\nonumber\\
	&=2h_{ip}h_{iq}\langle A_{jp},A_{jq}\rangle-2h_{ip}h_{jq}\langle A_{jp},A_{iq}\rangle\nonumber\\
	&=2|H|^{-2}\left(\langle A_{jp},A_{jq}\rangle\langle A_{ip},H\rangle\langle A_{iq},H\rangle-\langle A_{jp},A_{iq}\rangle\langle A_{ip},H\rangle\langle A_{jq},H\rangle\right).
	\end{align}Therefore,
	\begin{align*}
	|R_{ij}^\bot (\nu_1)-\bar{R}_{ij}(\nu_1)|^2&= 2|H|^{-2}\left(\langle A_{jp},A_{jq}\rangle\langle A_{ip},H\rangle\langle A_{iq},H\rangle-\langle A_{jp},A_{iq}\rangle\langle A_{ip},H\rangle\langle A_{jq},H\rangle\right).
	\end{align*}After reindexing (e.g. $j \to p\to q\to i\to j$ on the second term and $j\to i\to q\to j, p\to p$ on the third term), this gives
	\begin{align*}
	2|R_{ij}^\bot (\nu_1)-\bar{R}_{ij}(\nu_1)|^2&=4|H|^{-2}\left(\langle A_{ip},A_{pq}\rangle\langle A_{jq},H\rangle\langle A_{ij},H\rangle-\langle A_{ip},A_{jq}\rangle\langle A_{pq},H\rangle\langle A_{ij},H\rangle\right).	\end{align*}Thus, we have shown the reaction terms of our lemma statement are correct. For the gradient terms, it follows from the identities
	\begin{align}
	&\langle\nabla^\bot_k A_{ij},\nu_1\rangle=\langle\nabla^\bot_k A_{ij}^- ,\nu_1\rangle+\nabla_k h_{ij}, \nonumber\\
	&\langle\nabla^\bot_k A_{ij},\nabla^\bot_k \nu_1\rangle=\langle\nabla^\bot_k A_{ij}^- ,\nabla^\bot_k \nu_1\rangle+h_{ij}|\nabla_k^\bot \nu_1|^2, 	\label{7eqn4}\\
	&\nabla^\bot_k H=\nabla_k |H|\nu_1+|H|\nabla^\bot_k \nu_1. \nonumber
	\end{align}Therefore, we have
\begin{align}\label{7eqn5}
	&-4|H|^{-2}\langle\nabla^\bot_k A_{ij},\nabla^\bot_k H\rangle\langle A_{ij},H\rangle=-4|H|^{-1}h_{ij} \nabla_k |H|\langle\nabla^\bot_k A_{ij},\nu_1\rangle-4h_{ij}\langle\nabla^\bot_k A_{ij},\nabla^\bot_k \nu_1\rangle\nonumber\\
	&\ \quad\quad\quad\quad\quad\quad\quad\quad\quad\quad\quad\quad\quad\quad=-4|H|^{-1}\mathring{h}_{ij}\nabla_k |H|\langle\nabla^\bot_k A_{ij}^- ,\nu_1\rangle-4|H|^{-1}h_{ij} \nabla_k |H|\nabla_k h_{ij}\nonumber\\
	&\ \quad\quad\quad\quad\quad\quad\quad\quad\quad\quad\quad\quad\quad\quad-4\mathring{h}_{ij}\langle\nabla^\bot_k A_{ij}^- ,\nabla^\bot_k \nu_1\rangle-4|h|^2|\nabla_k^\bot\nu_1|^2,\nonumber\\
&-2|H|^{-2}|\nabla\langle A_{ij},H\rangle|^2=-2|H|^{-2}|\nabla(|H|h_{ij})|^2\nonumber\\
	&\ \ \quad\quad\quad\quad\quad\quad\quad\quad\quad=-2|H|^{-2}|h|^2|\nabla|H||^2-2|\nabla h|^2-4|H|^{-1}h_{ij} \nabla_k |H|\nabla_k h_{ij},
	\end{align}
since $A^-_{ii}=0$, meaning that it's trace free. Combining \eqref{7eqn2}-\eqref{7eqn5}, we get the desired result.
\end{proof}
Substituting the result of the above lemma into the evolution equation of $|H|^{-2}\sum_{i,j}|\langle A_{ij},H\rangle|^2$ and combining like terms, we have
	\begin{align*}
	\left(\partial_t-\Delta\right)\frac{\sum_{i,j}|\langle A_{ij},H\rangle|^2}{|H|^2}&=4\Big|\sum_{i,j}\mathring{h}_{ij}A_{ij}^- \Big|^2+2\sum_{i,j}|R_{ij}^\bot (\nu_1)-\bar{R}_{ij}(\nu_1)|^2+2|h|^4\\
	&-4|H|^{-1}\sum_{i,j,k}\mathring{h}_{ij}\nabla_k |H|\langle\nabla^\bot_k A_{ij}^- ,\nu_1\rangle-4\sum_{i,j,k}\mathring{h}_{ij}\langle\nabla^\bot_k A_{ij}^- ,\nabla^\bot_k\nu_1\rangle\\
	&-2|h|^2\sum_k|\nabla^\bot_k\nu_1|^2-2|\nabla h|^2+2B'-2|h|^2|H|^{-2}\sum_{k,\alpha,\beta}\bar{R}_{k\alpha k\beta} H^\alpha H^\beta.
	\end{align*}We negate the expression above, add in the evolution equation of $|A|^2$ and use \eqref{7eqn_P_a} to get
	\begin{align*}
	\left(\partial_t-\Delta\right)&|A^-|^2=-2|\nabla^\bot A|^2+2\sum_{i,j,p,q}|\langle A_{ij},A_{pq}\rangle|^2+2\sum_{i,j}|R_{ij}^\bot -\bar{R}^\bot_{ij}|^2+\left(\mathcal{P}-2B'\right)-4\Big|\sum_{i,j}\mathring{h}_{ij}A_{ij}^- \Big|^2\\
	&-2\sum_{i,j}|R_{ij}^\bot (\nu_1)-\bar{R}_{ij}(\nu_1)|^2-2|h|^4+4|H|^{-1}\sum_{i,j,k}\mathring{h}_{ij}\nabla_k |H|\langle\nabla^\bot_k A_{ij}^- ,\nu_1\rangle+2|h|^2\sum_k |\nabla^\bot_k\nu_1|^2\\
	&+4\sum_{i,j,k}\mathring{h}_{ij}\langle\nabla^\bot_k A_{ij}^- ,\nabla^\bot_k \nu_1\rangle+2|\nabla h|^2+2|h|^2|H|^{-2}\sum_{k,\alpha,\beta}\bar{R}_{k\alpha k\beta} H^\alpha H^\beta,
	\end{align*}
where $P$ here is
	\begin{align*}
	\mathcal{P}&= 4\sum_{i,j,p,q}\bar{R}_{ipjq}\left(\sum_{\alpha} A^\alpha_{pq}A^\alpha_{ij}\right)-4\sum_{j,k,p}\bar{R}_{kjkp}\left(\sum_{i,\alpha} A^\alpha_{pi}A^\alpha_{ij}\right)+2\sum_{k,\alpha,\beta}\bar{R}_{k\alpha k\beta}\left(\sum_{i,j} A^\alpha_{ij}A_{ij}^\beta \right)\nonumber\\
	&-8\sum_{j,p,\alpha,\beta}\bar{R}_{jp\alpha\beta}\left(\sum_iA^\alpha_{ip}A_{ij}^\beta \right)
	\end{align*} the lower order terms appearing in \eqref{eqn_|A|^2}.	
Taking the term $2|H|^{-2}\sum_{i,j,k,\alpha,\beta}\bar{R}_{k\alpha k\beta} H^\alpha A_{ij}^\beta \langle A_{ij},H\rangle$ out of $2B'$ and the last term of the evolution equation of $\frac{\sum_{i,j}|\langle A_{ij},H\rangle|^2}{|H|^2}$, we have
	\begin{align*}
	&2|H|^{-2}\sum_{i,j,k,\alpha,\beta}\bar{R}_{k\alpha k\beta}H^\alpha A_{ij}^\beta \langle A_{ij},H\rangle-2|h|^2|H|^{-2}\sum_{i,j,k,\alpha,\beta} \bar{R}_{k\alpha k\beta} H^\alpha H^\beta\\
	&=2|H|^{-2}\sum_{i,j,k,\alpha,\beta}\bar{R}_{k\alpha k\beta} H^\alpha\left(A^{-,\beta}_{ij}+\frac{|\langle A_{ij},H\rangle|}{|H|^2}H^\beta\right) \langle A_{ij},H\rangle-2|h|^2|H|^{-2}\sum_{i,j,k,\alpha,\beta} \bar{R}_{k\alpha k\beta} H^\alpha H^\beta\\
	&=2|H|^{-2}\sum_{i,j,k,\alpha,\beta\ge 2}\bar{R}_{k\alpha k\beta} H^\alpha A^{\beta}_{ij} \langle A_{ij},H\rangle.
	\end{align*}The reaction terms satisfy
	\begin{align*}
	&2\sum_{i,j,p,q}|\langle A_{ij},A_{pq}\rangle|^2-4\Big|\sum_{i,j}\mathring{h}_{ij}A_{ij}^- \Big|^2-2|h|^4=2\sum_{i,j,p,q}|\langle A_{ij}^- ,A_{pq}^- \rangle|^2,\nonumber\\
	&2\sum_{i,j}|R_{ij}^\bot-\bar{R}^\perp_{ij} |^2-2\sum_{i,j}|R_{ij}^\bot (\nu_1)-\bar{R}_{ij}(\nu_1)|^2=2|\hat{R}^\bot-\bar{R}|^2+2\sum_{i,j}|R_{ij}^\bot (\nu_1)-\bar{R}_{ij}(\nu_1)|^2, 
	\end{align*}
where
	\begin{align*}
	|\hat{R}^\bot-\bar{R}|^2=\sum_{i,j,p,\alpha,\beta\ge 2}|A_{ip}^\alpha A_{jp}^\beta -A_{jp}^\alpha A_{ip}^\beta |^2.
	\end{align*}
As for the gradient terms, taking the form of $\nabla^\bot_i A_{jp}=\nabla^\bot_i A_{jp}^- +\nabla_i h_{jp}\nu_1+h_{jp}\nabla^\bot_i \nu_1$, we see
	\begin{align*}
	|\nabla^\bot A|^2=|\nabla^\bot A^-|^2+|\nabla h|^2+|h|^2|\nabla^\bot \nu_1|^2+2\sum_{i,j}\mathring{h}_{ij}\langle\nabla^\bot_k A_{ij}^- ,\nabla^\bot_k \nu_1\rangle+2\sum_{i,j,k}\nabla_k \mathring{h}_{ij}\langle\nabla^\bot_k A_{ij}^- ,\nu_1\rangle.
	\end{align*}Thus,
	\begin{align*}
	-2|\nabla^\bot A|^2+2|\nabla h|^2+2|h|^2|\nabla^\bot \nu_1|^2&+4\sum_{i,j,k}\mathring{h}_{ij}\langle\nabla^\bot_k A_{ij}^- ,\nabla^\bot_k \nu_1\rangle\\
	&=-2|\nabla^\bot A^-|^2-4\sum_{i,j,k}\nabla_k \mathring{h}_{ij}\langle\nabla^\bot_k A_{ij}^- ,\nu_1\rangle.
	\end{align*}Putting this all together gives
	\begin{align*}
	\left(\partial_t-\Delta\right)|A^-|^2&=2\sum_{i,j,p,q}|\langle A_{ij}^- ,A_{pq}^- \rangle|^2+2|\hat{R}^\bot-\bar{R}|^2+2\sum_{i,j}|R_{ij}^\bot (\nu_1)-\bar{R}_{ij}(\nu_1)|^2-2|\nabla^\bot A^-|^2\\
&+4|H|^{-1}\sum_{i,j,k}\mathring{h}_{ij}\nabla_k |H|\langle\nabla^\bot_k A_{ij}^- ,\nu_1\rangle-4\sum_{i,j,k}\nabla_k\mathring{h}_{ij}\langle\nabla^\bot_k A_{ij}^- ,\nu_1\rangle+\left(\mathcal{P}-2B''\right)\\
	&-2|H|^{-2}\sum_{i,j,k,\alpha,\beta\ge 2}\bar{R}_{k\alpha k\beta} H^\alpha A^{\beta}_{ij}\langle A_{ij},H\rangle,
	\end{align*}
where
	\begin{align*}
	B''&:=2|H|^{-2}\sum_{i,j,p,q}\bar{R}_{ipjq}\langle A_{pq},H\rangle\langle A_{ij},H\rangle-2|H|^{-2}\sum_{i,j,k,p}\bar{R}_{kjkp} \langle A_{pi},H\rangle\langle A_{ij},H\rangle\\
	&+|H|^{-2}\sum_{i,j,k,\alpha,\beta}A^\alpha_{ij}\bar{R}_{k\alpha k\beta}H^\beta\langle A_{ij},H\rangle-4|H|^{-2}\sum_{i,j,p,\alpha,\beta}A^\alpha_{ip}\bar{R}_{jp\alpha\beta} H^\beta \langle A_{ij},H\rangle.
	\end{align*}
Note that since $\langle A_{ij}^-, \nu_1 \rangle = 0 $, differentiating with respect to $\nabla_k$ gives
	\begin{align*}
	\langle\nabla^\bot_k A_{ij}^- ,\nu_1\rangle&=-\langle A_{ij}^- ,\nabla^\bot_k \nu_1\rangle=-\langle\mathring{A}_{ij},\nabla^\bot_k \nu_1\rangle.
	\end{align*}Also since $\mathring{h}_{ij} = \langle \mathring{A}_{ij}, \nu_1 \rangle$ and from the equation above, we get
	\begin{align*}
	\nabla_k\mathring{h}_{ij}&= \langle\nabla^\bot_k \mathring{A}_{ij},\nu_1\rangle + \langle\mathring{A}_{ij},\nabla^\bot_k \nu_1\rangle =\langle\nabla^\bot_k \mathring{A}_{ij},\nu_1\rangle-\langle\nabla^\bot_kA_{ij}^- ,\nu_1\rangle.
	\end{align*}To simplify our final expression, let us define the tensor
	\begin{align*}
	Q_{ijk}:=\langle\nabla^\bot_k\mathring{A}_{ij},\nu_1\rangle-\langle\nabla^\bot_k A_{ij}^- ,\nu_1\rangle-|H|^{-1}\mathring{h}_{ij}\nabla_k |H|.
	\end{align*}Here we have the lower order terms in the evolution equation for the evolution of $|A^-|^2$. We match them to the evolution of the pinching quantity $f>0$. About the term $\mathcal{P}-2B''$, we have
	\begin{align*}
	\mathcal{P}-2B''&=4\sum_{i,j,p,q}\bar{R}_{ipjq}\left(\sum_{\alpha\ge 2} A^\alpha_{pq}A^\alpha_{ij}\right)-4\sum_{j,k,p}\bar{R}_{kjkp}\left(\sum_{i,\alpha\ge 2} A^\alpha_{pi}A^\alpha_{ij}\right)\\
	&+2\sum_{k,\alpha}\sum_{\beta\ge2}\bar{R}_{k\alpha k\beta}\left(\sum_{i,j} A^\alpha_{ij}A_{ij}^\beta \right)-8\sum_{j,p,\alpha}\sum_{\beta\ge2}\bar{R}_{jp\alpha\beta}\left(\sum_iA^\alpha_{ip}A_{ij}^\beta \right).
	\end{align*}In conclusion, according to Theorem \ref{7B} and \eqref{7eqn_P_a}, we get the following proposition.
\begin{proposition}\label{eqnof|A^-|^2}
The evolution equation of $|A^-|^2$ is
	\begin{align*}
	\left(\partial_t-\Delta\right)|A^-|^2&=2\sum_{i,j,p,q}|\langle A_{ij}^- ,A_{pq}^- \rangle|^2+2|\hat{R}^\bot-\bar{R}|^2+2\sum_{i,j}|R_{ij}^\bot (\nu_1)-\bar{R}_{ij}(\nu_1)|^2-2|\nabla^\bot A^-|^2\\
&+\left(\mathcal{P}-2B''\right)-4\sum_{i,j,k}Q_{ijk}\langle A_{ij}^- ,\nabla^\bot_k \nu_1\rangle-2|H|^{-2}\sum_{i,j,k,\alpha,\beta\ge 2}\bar{R}_{k\alpha k\beta} H^\alpha A^{\beta}_{ij}\langle A_{ij},H\rangle,
	\end{align*}
where
	\begin{align*}
	Q_{ijk}:=\langle\nabla^\bot_k\mathring{A}_{ij},\nu_1\rangle-\langle\nabla^\bot_k A_{ij}^- ,\nu_1\rangle-|H|^{-1}\mathring{h}_{ij}\nabla_k |H|.
	\end{align*}\end{proposition}
We consider the function $f=d_n+c_n |H|^2-|A|^2$. The assumption of the theorem is $f>0$ everywhere on $\mathcal{M}_0$. As $\mathcal{M}_0$ is compact, there exist constants $\e_0,\e_1>0$ depending on $\mathcal{M}_0$, such that $f\ge\e_1|H|^2+\e_0$, on $\mathcal{M}_0$. By Theorem \ref{ThCPn}, $f\ge\e_1|H|^2+\e_0$, on $\mathcal{M}_t$, for every $t\in[0,T)$. Recall, for codimension $m=1$, so far, we have that
	\begin{align*}
	&c_n=\frac{1}{n-\tilde{k}} \ \text{ and} \ \ d_n=2\tilde{k}, \ \text{ for} \ \ n\ge\max\left\{9,\frac{18+16\tilde{k}}{7}\right\},\ \ \tilde{k}\ge 1,\ \ 1=2k-n.
	\end{align*}
For higher codimension, we have 
	\begin{align*}
&c_n=\frac{1}{n-\tilde{k}} \ \text{ and} \ \ d_n=\frac{\tilde{k}(n-3-4m)}{n}, \ \text{ for} \ \ n\ge\max\{12,4\tilde{k}\}, \ \tilde{k}\ge 1, \ \ 2\le m<\frac{n-3}{4}.
	\end{align*}

Let $\delta>0$ be a small constant to be determined later in the proof. By previous work, the evolution equation for f is
	\begin{align*}
	\left(\partial_t-\Delta\right)f&=2(|\nabla^\bot A|^2-c_n|\nabla^\bot H|^2)+2\left(c_n\sum_{i,j}|\langle A_{ij},H\rangle|^2-\sum_{i,j,p,q}|\langle A_{ij},A_{pq}\rangle|^2-\sum_{i,j}|R_{ij}^\bot -\bar{R}^\bot_{ij}|^2\right)\nonumber\\
	&+2c_n\sum_{k,\alpha,\beta} \bar{R}_{k\alpha k\beta} H^\alpha H^\beta-\mathcal{P}.
	\end{align*}
We let $C$ denote an arbitrary constant depending upon the dimension $n,K_1,K_2$ and $d_n$, which may change from line to line. The pinching condition implies both terms on the first row on the right hand side of the equation for $f$ are non negative at each point in space-time. The first step of the proof and the main effort is to analyse the evolution equation $\frac{|A^-|^2}{f}$. We will show this ratio satisfies a favourable evolution equation with a right hand side has a nonpositive term. Specifically, we will show that
	\begin{align}\label{7initialclaim}
	\left(\partial_t-\Delta\right)\frac{|A^-|^2}{f}&\le2\left\langle\nabla\frac{|A^-|^2}{f},\nabla\log f\right\rangle-\delta\frac{|A^-|^2}{f^2}\left(\partial_t-\Delta\right)f+C,
	\end{align}for $C$ constant, that depends on $n,K_1,K_2$ and $d_n$. Then, since at the limit the background space is Euclidean, the result will follow from the maximum principle. By what we have shown this far, the evolution equation of $\frac{|A^-|^2}{f}$ is
	\begin{align*}
&(\partial_t-\Delta)\frac{|A^-|^2}{f}=\frac{1}{f}\left(\partial_t-\Delta\right)|A^-|^2-|A^-|^2\frac{1}{f^2}\left(\partial_t-\Delta\right)f+2\left\langle\nabla\frac{|A^-|^2}{f},\nabla \log f\right\rangle\\
	&=\frac{2}{f}\left(\sum_{i,j,p,q}|\langle A_{ij}^- ,A_{pq}^- \rangle|^2+|\hat{R}^\bot-\bar{R}|^2+\sum_{i,j}|R_{ij}^\bot (\nu_1)-\bar{R}_{ij}(\nu_1)|^2\right)\\
	&+\frac{2}{f}\left(-\frac{|A^-|^2}{f}\left(c_n\sum_{i,j}|\langle A_{ij},H\rangle|^2-\sum_{i,j,p,q}|\langle A_{ij},A_{pq}\rangle|^2-\sum_{i,j}|R_{ij}^\bot-\bar{R}^\bot_{ij} |^2\right)\right)\\
	&+\frac{2}{f}\left(2\sum_{i,j,p,q}\bar{R}_{ipjq}\left(\sum_{\alpha\ge 2} A^\alpha_{pq}A^\alpha_{ij}\right)-2\sum_{j,k,p}\bar{R}_{kjkp}\left(\sum_{i,\alpha\ge 2} A^\alpha_{pi}A^\alpha_{ij}\right)+\sum_{k,\alpha,\beta\ge 2}\bar{R}_{k\alpha k\beta}\left(\sum_{i,j} A^\alpha_{ij}A_{ij}^\beta \right)\right)\\
	&+\frac{2}{f}\left(-|H|^{-2}\sum_{i,j,k,\alpha}\sum_{\beta\ge2}\bar{R}_{k\alpha k\beta} H^\alpha A^{\beta}_{ij}\langle A_{ij},H\rangle-4\sum_{j,p,\alpha}\sum_{\beta\ge2}\bar{R}_{jp\alpha \beta}\left(\sum_iA^\alpha_{ip}A_{ij}^\beta \right)\right)\\
	&+\frac{2}{f}\left(-2\sum_{i,j,k}Q_{ijk}\langle A_{ij}^- ,\nabla^\bot_k \nu_1\rangle-|\nabla^\bot A^-|^2-\frac{|A^-|^2}{f}(|\nabla^\bot A|^2-c_n|\nabla^\bot H|^2)\right)\\
	&+\frac{2}{f}\left(\frac{|A^-|^2}{f}\left(2\sum_{i,j,p,q}\bar{R}_{ipjq}\left(\sum_{\alpha} A^\alpha_{pq}A^\alpha_{ij}\right)-2\sum_{j,k,p}\bar{R}_{kjkp}\left(\sum_{i,\alpha} A^\alpha_{pi}A^\alpha_{ij}\right)\right)\right)\\
	&+\frac{2}{f}\left(\frac{|A^-|^2}{f}\left(\sum_{k,\alpha,\beta}\bar{R}_{k\alpha k\beta}\left(\sum_{i,j} A^\alpha_{ij}A_{ij}^\beta \right)\right)\right)\\
	&+\frac{2}{f}\left(\frac{|A^-|^2}{f}\left(-c_n\sum_{k,\alpha,\beta}\bar{R}_{k\alpha k\beta}H^\alpha H^\beta-4\sum_{j,p,\alpha,\beta}\bar{R}_{jp\alpha\beta}\left(\sum_iA^\alpha_{ip}A_{ij}^\beta \right)\right)\right)+2\left\langle \nabla\frac{|A^-|^2}{f},\nabla\log f\right\rangle.
	\end{align*}
Let us give a brief explanation of the above evolution equation. The first two lines on the right hand side are the higher order terms. The terms in the third and fourth line are lower order terms, that are orthogonal to the principal direction. The terms on the fifth and the last term in the eighth line are gradient terms and the terms on the sixth to the eightth line are lower order terms, both in the principal direction and orthogonal to the principal direction.
\subsection{Codimension estimate}\label{section8.2}
We begin by estimating the reaction terms. We will make use of two estimates. The first estimate is proven on page 372 in \cite{AnBa10}, Section 3. The second estimate is a matrix inequality, which is Lemma 3.3 in \cite{Li1992}.
\begin{lemma}\label{4.1}
	\begin{align}\label{eq4.5}
	\sum_{i, j}\big|\mathring{h}_{ij} A^-_{ij}\big|^2+\sum_{i,j}|R_{ij}^{\perp}(\nu_1)-\bar{R}_{ij}(\nu_1)|^2 \leq 2|\mathring{h}|^2|{A^-}|^2,\\
\label{eq4.6}
	\sum_{i,j,p,q}|\langle A^-_{ij}, A^-_{pq}\rangle|^2+|\hat{R}^\perp-\bar{R}|^2 \leq\frac{3}{2}|A^-|^4.
	\end{align}
\end{lemma}
As an immediate consequence of the previous lemma, we have the following estimate for the reaction terms coming from the evolution of $|A^-|^2$.
\begin{lemma} [Upper bound for the reaction terms of $(\partial_t-\Delta)|A^-|^2$]\label{4.2}
	\begin{align}\label{eq4.10}
	\sum_{i,j,p,q}|\langle A^-_{i j}, A^-_{pq}\rangle|^2+|\hat{R}^{\perp}-\bar{R}|^2+\sum_{i,j,}|R_{i j}^{\perp}(\nu_1)-\bar{R}_{ij}(\nu_1)|^2&\leq\frac{3}{2}|A^-|^4+2|\mathring{h}|^2|A^-|^2.
	\end{align}
\end{lemma}
\begin{proof} The proof follows from Lemma \ref{4.1}.
\end{proof}
We express the reaction term in the evolution of $f$ in terms of $A^-, \mathring{h}$, and $|H|$. In view of the definition of $f$, observe that
	\begin{align}\label{7eq4.11}
	\frac{n c_n-1}{n}|H|^2=|A^-|^2+|\mathring{h}|^2+f-d_n.
	\end{align}In the following lemma, we get a lower bound for the reaction terms in the evolution of $f$.
\begin{lemma} [Lower bound for the reaction terms of $\left(\partial_t-\Delta\right) f$]\label{7lemma4.3} \ \newline
Consider $\tilde{k}\ge1$. If  $c_n = \frac{1}{n-\tilde{k}}$ and $n\ge4\tilde{k}$, then
	\begin{align}\label{7eq4.12}
	\frac{|A^-|^2}{f}&\left(c_n\sum_{i,j}\left|\left\langle A_{i j}, H\right\rangle\right|^2-\sum_{i,j,p,q}\left|\left\langle A_{i j}, A_{pq}\right\rangle\right|^2-\sum_{i,j}|R_{i j}^{\perp}-\bar{R}^\bot_{ij}|^2\right)\geq \frac{2}{n c_n-1}|A^-|^4+\frac{n c_n}{n c_n-1}|\mathring{h}|^2|A^-|^2\nonumber\\
	&-\frac{|A^-|^2}{f}\left(\frac{2}{nc_n-1}d_n\left(|A^-|^2+f\right)+\frac{nc_n}{nc_n-1}|\mathring{h}|^2d_n\right).
	\end{align}\end{lemma}
\begin{proof} We do a computation that is similar to a computation in \cite{AnBa10}, except we do not throw away the pinching term $f$. By the following equations
	\begin{align*}
	&|h|^2=|\mathring{h}|^2+\frac{1}{n}|H|^2,\nonumber\\
	&\sum_{i,j}|\langle A_{ij},H\rangle|^2=|H|^2|h|^2,\nonumber\\
	&\sum_{i,j,p,q}|\langle A_{ij},A_{pq}\rangle|^2=|h|^4+2\sum_{i,j}\big| \mathring{h}_{ij}A^-_{ij}\big|^2+\sum_{i,j,p,q}|\langle A^-_{ij},A^-_{pq}\rangle|^2,\nonumber\\
	&\sum_{i,j}|R_{ij}^{\perp}-\bar{R}^\perp_{ij}|^2=|\hat{R}^{\perp}-\bar{R}|^2+2\sum_{i,j}|R_{i j}^{\perp}\left(\nu_1\right)-\bar{R}_{ij}(\nu_1)|^2\nonumber,
	\end{align*}we have
	\begin{align*}
	&c_n\sum_{i,j}\left|\left\langle A_{i j}, H\right\rangle\right|^2-\sum_{i,j,p,q}\left|\left\langle A_{i j}, A_{pq}\right\rangle\right|^2-\sum_{i,j}|R_{i j}^{\perp}-\bar{R}^\bot_{ij}|^2=\frac{1}{n} c_n|H|^4+c_n|\mathring{h}|^2|H|^2-|\mathring{h}|^4-|\hat{R}^{\perp}-\bar{R}^\bot|^2\\
	&-\frac{2}{n}|\mathring{h}|^2|H|^2-\frac{1}{n^2}|H|^4-2\sum_{i,j}|\mathring{h}_{i j} A^-_{i j}|^2-\sum_{i,j,p,q}|\langle A^-_{i j}, A^-_{pq}\rangle|^2-2\sum_{i,j}|R_{i j}^{\perp}(\nu_1)-\bar{R}_{ij}(\nu_1)|^2\\
	&=\frac{1}{n}\left(c_n-\frac{1}{n}\right)|H|^4+\left(c_n-\frac{1}{n}\right)|\mathring{h}|^2|H|^2-\frac{1}{n}|\mathring{h}|^2|H|^2-|\mathring{h}|^4\\
	&-2\sum_{i,j}|\mathring{h}_{i j} A^-_{i j}|^2-2\sum_{i,j}|R_{i j}^{\perp}(\nu_1)-\bar{R}_{ij}(\nu_1)|^2-\sum_{i,j,p,q}|\langle A^-_{i j}, A^-_{pq}\rangle|^2-|\hat{R}^{\perp}-\bar{R}|^2.
	\end{align*}
Use \eqref{7eq4.11} to substitute in for $|H|^2$ and cancelling terms we get
	\begin{align*}
	c_n&\sum_{i,j}\left|\left\langle A_{i j}, H\right\rangle\right|^2-\sum_{i,j,p,q}\left|\left\langle A_{i j}, A_{pq}\right\rangle\right|^2-\sum_{i,j}|R_{i j}^{\perp}-\bar{R}^\bot_{ij}|^2\\
	&=\frac{1}{n}\left(|A^-|^2+|\mathring{h}|^2+f-d_n \right)|H|^2+|\mathring{h}|^2\left( |A^-|^2+|\mathring{h}|^2+f-d_n\right)\\
	&-\frac{1}{n}|\mathring{h}|^2|H|^2-|\mathring{h}|^4-2\sum_{i,j}|\mathring{h}_{i j} A^-_{i j}|^2-2\sum_{i,j}|R_{i j}^{\perp}(\nu_1)-\bar{R}_{ij}(\nu_1)|^2-\sum_{i,j,p,q}|\langle A^-_{i j}, A^-_{pq}\rangle|^2-|\hat{R}^{\perp}-\bar{R}|^2\\
	&=\frac{1}{n}\left(f+|A^-|^2-d_n\right)|H|^2+\left(f+|A^-|^2-d_n\right) |\mathring{h}|^2\\
	&-2\sum_{i,j}|\mathring{h}_{i j} A^-_{i j}|^2-2\sum_{i,j}|R_{i j}^{\perp}(\nu_1)-\bar{R}_{ij}(\nu_1)|^2-\sum_{i,j,p,q}|\langle A^-_{i j}, A^-_{pq}\rangle|^2-|\hat{R}^{\perp}-\bar{R}|^2.
	\end{align*}Using \eqref{7eq4.11} once more for the remaining factor of $|H|^2$ gives
	\begin{align*}
	&c_n\sum_{i,j}\left|\left\langle A_{i j}, H\right\rangle\right|^2-\sum_{i,j,p,q}\left|\left\langle A_{i j}, A_{pq}\right\rangle\right|^2-\sum_{i,j}|R_{i j}^{\perp}-\bar{R}^\bot_{ij}|^2 \\
	&= \frac{1}{n}\left(f+|A^-|^2-d_n\right)\left(c_n-\frac{1}{n}\right)^{-1}(f+|A^-|^2+|\mathring{h}|^2-d_n)+\left(f+|A^-|^2-d_n\right)|\mathring{h}|^2 \\
	&-2\sum_{i,j}|\mathring{h}_{i j} A^-_{i j}|^2-2\sum_{i,j}|R_{i j}^{\perp}(\nu_1)-\bar{R}_{ij}(\nu_1)|^2-\sum_{i,j,p,q}|\langle A^-_{i j}, A^-_{pq}\rangle|^2-|\hat{R}^{\perp}-\bar{R}|^2 \\
	&= \frac{1}{nc_n-1} f(f+2|A^-|^2+|\mathring{h}|^2-2d_n)+f|\mathring{h}|^2+\frac{1}{nc_n-1}| A^-|^4+\frac{nc_n}{nc_n-1}|A^-|^2|\mathring{h}|^2 -|\hat{R}^{\perp}-\bar{R}|^2\\
	&-\frac{nc_n}{nc_n-1}d_n |\mathring{h}|^2-\frac{2}{nc_n-1}d_n|A^-|^2+\frac{1}{nc_n-1}d_n^2-2\sum_{i,j}|\mathring{h}_{i j} A^-_{i j}|^2\\
	&-2\sum_{i,j}|R_{i j}^{\perp}(\nu_1)-\bar{R}_{ij}(\nu_1)|^2-\sum_{i,j,p,q}|\langle A^-_{i j}, A^-_{pq}\rangle|^2.
	\end{align*}Now by the two estimates in Lemma \ref{4.1},
	\begin{align*}
	2\sum_{i,j}|\mathring{h}_{ij}A^-_{ij}|^2&+2\sum_{i,j}|R^\bot_{ij}(\nu_1)-\bar{R}_{ij}(\nu_1)|^2+\sum_{i,j,p,q}|\langle A^-_{ij},A^-_{pq}\rangle|^2+|\hat{R}^\bot-\bar{R}|^2\le4|\mathring{h}|^2|A^-|^2+\frac{3}{2}|A^-|^4.
	\end{align*}Therefore,
	\begin{align*}
	\frac{1}{nc_n-1}&|A^-|^4+\frac{nc_n}{nc_n-1}|A^-|^2|\mathring{h}|^2-2\sum_{i,j}|\mathring{h}_{i j} A^-_{i j}|^2-2\sum_{i,j}|R_{i j}^{\perp}(\nu_1)-\bar{R}_{ij}(\nu_1)|^2\\
	&-\sum_{i,j,p,q}|\langle A^-_{i j}, A^-_{pq}\rangle|^2-|\hat{R}^{\perp}-\bar{R}|^2\geq\left(\frac{1}{nc_n-1}-\frac{3}{2}\right)|A^-|^4+\left(\frac{nc_n}{nc_n-1}-4\right)|\mathring{h}|^2|A^-|^2.
	\end{align*}Since $c_n = \frac{1}{n-\tilde{k}}$ and $n\ge 4\tilde{k}$, we have $\frac{1}{n c_n-1}-\frac{3}{2} \geq \frac{3}{2}$ and $\frac{nc_n}{nc_n-1}-4 \geq 0.$
Consequently, we have
	\begin{align}\label{7eq4.13}
	c_n\sum_{i,j}&\left|\left\langle A_{i j}, H\right\rangle\right|^2-\sum_{i,j,p,q}\left|\left\langle A_{i j}, A_{pq}\right\rangle\right|^2-\sum_{i,j}|R_{i j}^{\perp}-\bar{R}|^2 \geq \frac{2}{nc_n-1} f|A^-|^2+\frac{nc_n}{nc_n-1} f|\mathring{h}|^2\nonumber\\
	&+\frac{1}{nc_n-1}d_n^2+\frac{1}{nc_n-1} f^2-\frac{2}{nc_n-1}d_n\left(|A^-|^2+f\right)-\frac{nc_n}{nc_n-1}|\mathring{h}|^2d_n\nonumber\\
	&\geq \frac{2}{nc_n-1} f|A^-|^2+\frac{nc_n}{nc_n-1} f|\mathring{h}|^2-\frac{2}{nc_n-1}d_n\left(|A^-|^2+f\right)-\frac{nc_n}{nc_n-1}|\mathring{h}|^2d_n.
	\end{align}Multiplying both sides by $\frac{|A^-|^2}{f}$ completes the proof of the lemma.
\end{proof}
Putting Lemmas \ref{4.2} and \ref{7lemma4.3} together, we have the following lemma.
\begin{lemma}[Reaction term estimate]\label{7lemma4.4} \ \newline
Consider $\tilde{k}\ge1$ and $0<\delta \leq \frac{1}{2}$, in which case $c_n= \frac{1}{n-\tilde{k}}$ and $n\ge 4\tilde{k}$. Then
	\begin{align}\label{7eq4.14}
	&\sum_{i,j,p,q}|\langle A^-_{i j}, A^-_{pq}\rangle|^2+|\hat{R}^{\perp}|^2+\sum_{i,j}|R_{i j}^{\perp}(\nu_1)-\bar{R}_{ij}(\nu_1)|^2\leq(1-\delta) \frac{|A^-|^2}{f}\Big(c_n\sum_{i,j}\left|\left\langle A_{i j}, H\right\rangle\right|^2\nonumber\\
	&-\sum_{i,j,p,q}\left|\left\langle A_{i j}, A_{pq}\right\rangle\right|^2-\sum_{i,j}|R_{i j}^{\perp}-\bar{R}^\bot_{ij}|^2\Big)\nonumber\\
	&+(1-\delta)\frac{|A^-|^2}{f}\left(\frac{2}{nc_n-1}d_n\left(|A^-|^2+f\right)+\frac{nc_n}{nc_n-1}|\mathring{h}|^2d_n\right).
	\end{align}\end{lemma}
\begin{proof}
In view of \eqref{eq4.10} and \eqref{7eq4.12}, we have
	\begin{align*}
	&\sum_{i,j,p,q}|\langle A^-_{i j}, A^-_{pq}\rangle|^2+|\hat{R}^{\perp}-\bar{R}|^2+\sum_{i,j}|R_{i j}^{\perp}(\nu_1)-\bar{R}_{ij}(\nu_1)|^2-(1-\delta) \frac{|A^-|^2}{f}\Big(c_n\sum_{i,j}\left|\left\langle A_{i j}, H\right\rangle\right|^2\\
	&-\sum_{i,j,p,q}\left|\left\langle A_{i j}, A_{pq}\right\rangle\right|^2-\sum_{i,j}|R_{i j}^{\perp}-\bar{R}^\bot_{ij}|^2\Big) \\
	&\leq \frac{3}{2}|A^-|^4+2|\mathring{h}|^2|A^-|^2-\frac{2(1-\delta)}{n c_n-1}|A^-|^4-\frac{n c_n(1-\delta)}{n c_n-1}|\mathring{h}|^2|A^-|^2 \\
	&+(1-\delta)\frac{|A^-|^2}{f}\left(\frac{2}{nc_n-1}d_n\left(|A^-|^2+f\right)+\frac{nc_n}{nc_n-1}|\mathring{h}|^2d_n\right)\\
	&=\left(\frac{3}{2}-\frac{2(1-\delta)}{n c_n-1}\right)|A^-|^4+\left(2-\frac{n c_n(1-\delta)}{n c_n-1}\right)|\mathring{h}|^2|A^-|^2\\
	&+(1-\delta)\frac{|A^-|^2}{f}\left(\frac{2}{nc_n-1}d_n\left(|A^-|^2+f\right)+\frac{nc_n}{nc_n-1}|\mathring{h}|^2d_n\right).
	\end{align*}Since  $c_n = \frac{1}{n-\tilde{k}}$ and $n\ge 4\tilde{k}$, then $\frac{1}{n c_n-1} \geq 3$ and $\frac{n c_n}{n c_n-1} \geq 4.$ Therefore, if $\delta\le\frac{1}{2}$
	\begin{align*}
	\frac{3}{2}-\frac{2(1-\delta)}{n c_n-1}\le \frac{3}{2}-6(1-\delta)\le 0 \ \ \text{and} \ \ 2-\frac{n c_n(1-\delta)}{nc_n-1}\le 2-4(1-\delta)\le0,
	\end{align*}
which gives \eqref{7eq4.14}.
\end{proof}
From \eqref{Berger} and Theorem \ref{ThCPn}, we have that
	\begin{align*}
	&\frac{|A^-|^2}{f^2}\left(4\sum_{i,j,p,q}\bar{R}_{ipjq}\left(\sum_{\alpha} A^\alpha_{pq}A^\alpha_{ij}\right)-4\sum_{j,k,p}\bar{R}_{kjkp}\left(\sum_{i,\alpha} A^\alpha_{pi}A^\alpha_{ij}\right)+2\sum_{k,\alpha,\beta}\bar{R}_{k\alpha k\beta}\left(\sum_{i,j} A^\alpha_{ij}A_{ij}^\beta \right)\right)\\
	&+\frac{|A^-|^2}{f^2}\left(-2c_n\sum_{k,\alpha,\beta}\bar{R}_{k\alpha k\beta}H^\alpha H^\beta-8\sum_{j,p,\alpha,\beta}\bar{R}_{jp\alpha\beta}\left(\sum_i A^\alpha_{ip}A_{ij}^\beta \right)\right)\\
	&\le C\frac{|A^-|^2}{f}\\
	&\le C,
	\end{align*}for $C$ constant, that depends on $n,K_1,K_2$ and $d_n$, which may change from line to line. Here, we used the fact that the quantities in the parenthesis divided by $f$ are bounded.
Also, from \eqref{Berger}, we have
	\begin{align*}
	&\frac{1}{f}\left(4\sum_{i,j,p,q}\bar{R}_{ipjq}\left(\sum_{\alpha\ge 2} A^\alpha_{pq}A^\alpha_{ij}\right)-4\sum_{j,k,p}\bar{R}_{kjkp}\left(\sum_{i,\alpha\ge 2} A^\alpha_{pi}A^\alpha_{ij}\right)+2\sum_{k,\alpha,\beta\ge 2}\bar{R}_{k\alpha k\beta}\left(\sum_{i,j} A^\alpha_{ij}A_{ij}^\beta \right)\right)\\
	&+\frac{1}{f}\left(-2|H|^{-2}\sum_{i,j,k,\alpha}\sum_{\beta\ge2}\bar{R}_{k\alpha k\beta} H^\alpha A^{\beta}_{ij}\langle A_{ij},H\rangle-8\sum_{j,p,\alpha}\sum_{\beta\ge2}\bar{R}_{jp\alpha \beta}\left(\sum_iA^\alpha_{ip}A_{ij}^\beta \right)\right)\\
	&\le C\frac{1}{f}\left(|A^-|^2 +|A^-||A|+|A^-||h|\right)\\
	&\le C,
	\end{align*}
where we used the fact that the quantities in the parenthesis satisfy inequalities of the form $|h|\le C\sqrt{f}$ and $|A|\le C \sqrt{f}$. By previous calculations we have upper bounds for most of the terms. We will now prove an estimate about the gradient terms.
\begin{lemma}[{Lower bound for Bochner term of $\left(\partial_t-\Delta\right)|A^-|^2$}]\label{74.6} \ \newline
Consider $\tilde{k}\ge1$. If $c_n =\frac{1}{n-\tilde{k}}$, then
	\begin{align*}
	2|\hat{\nabla}^{\perp} A^-|^2&\geq\left(\frac{7n-18}{9(n+2)\left(n c_n-1\right)}-2\right)|\mathring{h}|^2|\nabla^{\perp} \nu_1|^2+\frac{7n-18}{9(n+2)\left(n c_n-1\right)}(|A^-|^2+f-d_n)|\nabla^{\perp} \nu_1|^2.
	\end{align*}\end{lemma}
\begin{proof}
We begin by applying Young's inequality
	\begin{align*}
	\sum_{i,j,k}|\hat{\nabla}_i^{\perp} A^-_{j k}+\mathring{h}_{j k} \nabla_i^{\perp} \nu_1|^2&\leq 2|\hat{\nabla}^{\perp} A^-|^2+2|\mathring{h}|^2|\nabla^{\perp} \nu_1|^2.
	\end{align*}
Multiplying both sides of \eqref{7eq4.11} by $\frac{7n-18}{9(n+2)(n c_n-1)}$ gives
	\begin{align*}
	\frac{7n-18}{9n(n+2)}|H|^2=\frac{7n-18}{9(n+2)\left(n c_n-1\right)}\left(f+|A^-|^2+|\mathring{h}|^2-d_n\right).
	\end{align*}In view of \eqref{eq4.22naff} and Young's inequality, our observations give us that
	\begin{align*}
	\frac{7n-18}{9(n+2)\left(n c_n-1\right)}\left(f+|A^-|^2+|\mathring{h}|^2-d_n\right)|\nabla^{\perp} \nu_1|^2&\leq 2|\hat{\nabla}^{\perp} A^-|^2+2|\mathring{h}|^2|\nabla^{\perp} \nu_1|^2.
\end{align*}Subtracting the $|\mathring{h}||\nabla^{\perp} \nu_1|^2$ term on the right-hand side gives
	\begin{align}\label{7eqn_nablaA-}
	\frac{7n-18}{9(n+2)\left(nc_n-1\right)}(f+|A^-|^2&-d_n)|\nabla^{\perp} \nu_1|^2+\left(\frac{7n-18}{9(n+2)\left(n c_n-1\right)}-2\right)|\mathring{h}|^2|\nabla^{\perp} \nu_1|^2 \nonumber\\
	&\leq 2|\hat{\nabla}^{\perp}A^-|^2.
	\end{align}
The estimate of the lemma follows.
\end{proof}
\begin{lemma} [{Lower bound for Bochner term of $\left(\partial_t-\Delta\right) f$}]\label{74.7} \ \newline
Consider $\tilde{k}\ge1$. If $c_n = \frac{1}{n-\tilde{k}}$ and $n\ge\max\{9,\frac{18+16\tilde{k}}{7}\}$, then
	\begin{align*}
	&2 \frac{|A^-|^2}{f}\left(|\nabla^{\perp} A|^2-c_n|\nabla^{\perp} H|^2\right)\geq\left(2-\frac{18(n+2)\left(n c_n-1\right)}{7n-18}\right) \frac{|A^-|^2}{f}\sum_{i,j,k}|\langle\nabla_i^{\perp} \mathring{A}_{j k}, \nu_1\rangle|^2 \\
	&+\frac{2n}{n c_n-1}\left(\frac{16}{9(n+2)}-c_n\right)|A^-|^2|\nabla^{\perp} \nu_1|^2+\frac{2n}{n c_n-1}\left(c_n-\frac{16}{9(n+2)}\right)\frac{|A^-|^2}{f}d_n|\nabla^{\perp} \nu_1|^2.
	\end{align*}\end{lemma}
\begin{proof}
Using \eqref{2.22naff} and \eqref{2.23naff}, we have
	\begin{align*}
	|\nabla^{\perp} A|^2-c_n|\nabla^{\perp} H|^2&=\sum_{i,j,k}|\langle\nabla_i^{\perp} A^-_{j k}, \nu_1\rangle+\nabla_i h_{j k}|^2-c_n|\nabla| H \|^2 \\
	&+\sum_{i,j,k}|\hat{\nabla}_i^{\perp} A^-_{j k}+h_{j k} \nabla_i^{\perp} \nu_1|^2-c_n|H|^2|\nabla^{\perp} \nu_1|^2.
	\end{align*}Note that by \eqref{eq4.15}, \eqref{eq4.17} and \eqref{eq4.21naff} we have
	\begin{align*}
	\sum_{i,j,k}|\langle\nabla_i^{\perp} A^-_{j k}, \nu_1\rangle+\nabla_i h_{j k}|^2-&c_n|\nabla| H||^2=\sum_{i,j,k}|\langle\nabla_i^{\perp} \mathring{A}_{j k}, \nu_1\rangle|^2-\frac{n c_n-1}{n}|\nabla| H||^2\\
	&\geq\left(1-\frac{9(n+2)\left(nc_n-1\right)}{7n-18}\right)\sum_{i,j,k}|\langle\nabla_i^{\perp} \mathring{A}_{j k}, \nu_1\rangle|^2.
	\end{align*}In view of \eqref{7eq4.11} and\eqref{4.20naff9}, we have
	\begin{align*}
	\sum_{i,j,k}|\hat{\nabla}_i^\bot A^-_{jk}+h_{jk} \nabla_i^\bot\nu_1|^2&-c_n|H|^2|\nabla^\bot \nu_1|^2\geq\left(\frac{16}{9(n+2)}-c_n\right)|H|^2|\nabla^\bot\nu_1|^2 \\
	&\geq \frac{n}{n c_n-1}\left(\frac{16}{9(n+2)}-c_n\right) (f-d_n)|\nabla^\bot \nu_1|^2.
	\end{align*}Thus, by the three previous computations, we have the estimate of the lemma.
\end{proof}
\begin{lemma}[Upper bound for gradient term of $\left(\partial_t-\Delta\right)|A^-|^2$ ]\label{74.8} \ \newline
Consider $\tilde{k}\ge1$. If $c_n =\frac{1}{n-\tilde{k}}$ and $|H|$ large, then
	\begin{align*}
	4 \sum_{i,j,k}&Q_{i j k}\langle A^-_{i j}, \nabla_k^{\perp} \nu_1\rangle\leq 2 |\langle\nabla^{\perp} A^-, \nu_1\rangle|^2+\left(2 a_2+2 a_3\frac{18\tilde{k}(n+2)}{(7n-18)(n-\tilde{k})}\right) \frac{|A^-|^2}{f}|\langle\nabla^{\perp} \mathring{A}, \nu_1\rangle|^2 \nonumber\\
	&+2|A^-|^2|\nabla^{\perp} \nu_1|^2+\frac{2}{a_2} f|\nabla^{\perp} \nu_1|^2+\frac{2}{a_3}|\mathring{h}|^2|\nabla^{\perp} \nu_1|^2.
	\end{align*}
\end{lemma}
\begin{proof}
Using the definition of $Q_{ijk}$, we get
	\begin{align}\label{74.30}
	|Q| \leq|\langle\nabla^{\perp} \mathring{A}, \nu_1\rangle|+|\langle\nabla^{\perp}A^-, \nu_1\rangle|+|H|^{-1}|\mathring{h}||\nabla| H||.
	\end{align}It easily follows from the definition of $f$ that
	\begin{align*}
	f \leq\left(c_n-\frac{1}{n}\right)|H|^2+d_n \le\frac{\tilde{k}}{n(n-\tilde{k})}|H|^2+d_n.
	\end{align*}
For $|H|$ large, $f<\frac{2\tilde{k}}{n(n-\tilde{k})}|H|^2.$ Consequently, using \eqref{eq4.21naff}, we obtain
	\begin{align}\label{74.31}
	\frac{|A^-|^2}{|H|^2}|\nabla| H||^2&\leq \frac{9n(n+2)}{7n-18} \frac{2\tilde{k}}{n(n-\tilde{k})} \frac{|A^-|^2}{f}\sum_{i,k}|\langle\nabla^{\perp} \mathring{A}, \nu_1\rangle|^2\nonumber\\
&=\frac{18\tilde{k}(n+2)}{(7n-18)(n-\tilde{k})} \frac{|A^-|^2}{f}\sum_{i,k}|\langle\nabla^{\perp} \mathring{A}, \nu_1\rangle|^2.
	\end{align}Then, \eqref{74.30} gives
	\begin{align*}
	4\sum_{i,j,k} Q_{i j k}\langle A^-_{i j}, \nabla_k^{\perp} \nu_1\rangle&\leq 4|Q||\langle A^-, \nabla^{\perp} \nu_1\rangle| \\
	&\leq 4\left(|\langle\nabla^{\perp}\mathring{A}, \nu_1\rangle|+|\langle\nabla^{\perp} A^-, \nu_1\rangle|+|H|^{-1}|\mathring{h} || \nabla| H||\right)|A^-||\nabla^{\perp} \nu_1|.
	\end{align*}Now to each of these three summed terms above we apply Young's inequality with parameters $a_1, a_2, a_3>0$. Specifically, we have
	\begin{align*}
	&4|\langle\nabla^{\perp} A^-, \nu_1\rangle||A^-||\nabla^{\perp} \nu_1|\leq 2 a_1|\langle\nabla^{\perp} A^-, \nu_1\rangle|^2+\frac{2}{a_1}|A^-|^2|\nabla^{\perp} \nu_1|^2, \\
	&4|\langle\nabla^{\perp} \mathring{A}, \nu_1\rangle||A^-||\nabla^{\perp} \nu_1|=4|\langle\nabla^{\perp} \mathring{A}, \nu_1\rangle| \frac{|A^-|}{\sqrt{f}} f^{\frac{1}{2}}|\nabla^{\perp} \nu_1|\leq 2 a_2 \frac{|A^-|^2}{f}|\langle\nabla^{\perp} \mathring{A}, \nu_1\rangle|^2+\frac{2}{a_2} f|\nabla^{\perp} \nu_1|^2, \\
	&4|H|^{-1}|\mathring{h}||\nabla| H|||A^-||\nabla^{\perp} \nu_1|\leq 2 a_3 \frac{|A^-|^2}{|H|^2}\left|\nabla| H|\right|^2+\frac{2}{a_3}|\mathring{h}|^2 |\nabla^{\perp} \nu_1|^2 \\
	&\quad\quad\quad\quad\quad\quad\quad\quad\quad\quad\quad\quad\leq 2 a_3\frac{18\tilde{k}(n+2)}{(7n-18)(n-\tilde{k})} \frac{|A^-|^2}{f}|\langle\nabla^{\perp} \mathring{A}, \nu_1\rangle|^2+\frac{2}{a_3}|\mathring{h}|^2|\nabla^{\perp} \nu_1|^2.
	\end{align*}Note we used \eqref{74.31} in the last inequality. Hence
	\begin{align*}
	4 \sum_{i,j,k}Q_{i j k}\langle A^-_{i j}, \nabla_k^{\perp} \nu_1\rangle&\leq 2 a_1|\langle\nabla^{\perp} A^-, \nu_1\rangle|^2+\left(2 a_2+2 a_3\frac{18\tilde{k}(n+2)}{(7n-18)(n-\tilde{k})}\right) \frac{|A^-|^2}{f}|\langle\nabla^{\perp} \mathring{A}, \nu_1\rangle|^2 \nonumber\\
	&+\frac{2}{a_1}|A^-|^2|\nabla^{\perp} \nu_1|^2+\frac{2}{a_2} f|\nabla^{\perp} \nu_1|^2+\frac{2}{a_3}|\mathring{h}|^2|\nabla^{\perp} \nu_1|^2.
	\end{align*}
Setting $a_1=1$ and keeping $a_2$ and $a_3$ as they are for now, we get the desired result.
\end{proof}
Finally, putting the conclusions of Lemma \ref{74.6}, \ref{74.7} and \ref{74.8} together, we get the following result. For this last Lemma, we will consider the specific case of $\tilde{k}=2$, since we are interested in cylindrical singularities of the form $\mathbb{S}^{n-1}\times\mathbb{R}$, in order to implement surgery later.
\begin{lemma}[Gradient term estimate]\label{74.9} \ \newline
Consider $ \tilde{k}=2$, $n\ge23$ and $0<\delta\le\min\left\{\frac{1}{2},1-\frac{72(n+2)(7 n-18)}{(n-2)\left(7 n-18-36(n+2) \cdot \frac{2}{n-2}\right)\left(14 n-36-18(n+2) \cdot \frac{2}{n-2}\right)}\right\}$, in which case $c_n = \frac{1}{n-\tilde{k}}$ and $d_n=\frac{\tilde{k}}{n}(n-3-4m)$. Then,
	\begin{align*}
	&4 \sum_{i,j,k}Q_{i j k}\left\langle A^-_{i j}, \nabla_k^{\perp} \nu_1\right\rangle\le2|\nabla^{\perp} A^-|^2+2(1-\delta) \frac{|A^-|^2}{f}\left(|\nabla^{\perp} A|^2-c_n|\nabla^{\perp} H|^2\right)\\
	&+\frac{1}{2}\frac{7n-18}{9(n+2)\left(n c_n-1\right)}d_n|\nabla^\bot\nu_1|^2-\frac{2n}{n c_n-1}\left(c_n-\frac{16}{9(n+2)}\right)\frac{|A^-|^2}{f}d_n|\nabla^{\perp} \nu_1|^2.
	\end{align*}\end{lemma}
\begin{proof}
Expanding $|\nabla^{\perp} A^-|^2$ using \eqref{2.24naff} and using Lemma \ref{74.6},, gives us
	\begin{align*}
	&2|\nabla^{\perp} A^-|^2=2|\hat{\nabla}^{\perp} A^-|^2+2|\langle\nabla^{\perp}A^-, \nu_1\rangle|^2 \\
	&\geq 2|\langle\nabla^{\perp} A^-, \nu_1\rangle|^2+\left(\frac{7n-18}{9(n+2)\left(n c_n-1\right)}-2\right)|\mathring{h}|^2|\nabla^{\perp} \nu_1|^2\\
	&+\frac{7n-18}{9(n+2)\left(n c_n-1\right)}(|A^-|^2+f-d_n)|\nabla^{\perp} \nu_1|^2.
	\end{align*}Multiplying the result in Lemma \ref{74.7} by $(1-\delta)$ and using that $1-\delta \geq \frac{1}{2}$ on the coefficient of $|A^-|^2|\nabla^{\perp} \nu_1|^2$ gives
	\begin{align*}
		&2(1-\delta) \frac{|A^-|^2}{f}(|\nabla^{\perp} A|^2-c_n|\nabla^{\perp} H|^2)\geq (1-\delta)\left(2-\frac{18(n+2)\left(n c_n-1\right)}{7n-18}\right) \frac{|A^-|^2}{f}\sum_{i,j,k}|\langle\nabla_i^{\perp} \mathring{A}_{j k}, \nu_1\rangle|^2 \\
	&+\frac{n}{n c_n-1}\left(\frac{16}{9(n+2)}-c_n\right)|A^-|^2|\nabla^{\perp} \nu_1|^2+(1-\delta)\frac{2n}{n c_n-1}\left(c_n-\frac{16}{9(n+2)}\right)\frac{|A^-|^2}{f}d_n|\nabla^{\perp} \nu_1|^2.
		\end{align*}Putting these together, we get
	\begin{align*}
	&2|\nabla^{\perp} A^-|^2+2(1-\delta) \frac{|A^-|^2}{f}(|\nabla^{\perp} A|^2-c_n|\nabla^{\perp} H|^2) \geq 2|\langle\nabla^{\perp} A^-, \nu_1\rangle|^2\\
	&+\left(\frac{7n-18}{9(n+2)\left(n c_n-1\right)}-2\right)|\mathring{h}|^2|\nabla^{\perp} \nu_1|^2+\frac{7n-18}{9(n+2)\left(n c_n-1\right)}(|A^-|^2+f-d_n)|\nabla^{\perp} \nu_1|^2\\
	&+(1-\delta)\left(2-\frac{18(n+2)\left(n c_n-1\right)}{7n-18}\right) \frac{|A^-|^2}{f}\sum_{i,j,k}|\langle\nabla_i^{\perp} \mathring{A}_{j k}, \nu_1\rangle|^2 \\
	&+\frac{n}{n c_n-1}\left(\frac{16}{9(n+2)}-c_n\right)|A^-|^2|\nabla^{\perp} \nu_1|^2+\frac{2n}{n c_n-1}\left(c_n-\frac{16}{9(n+2)}\right)\frac{|A^-|^2}{f}d_n|\nabla^{\perp} \nu_1|^2.
	\end{align*}
On the other hand, the first result of Lemma \ref{74.8} gives us that
	\begin{align*}
	4 \sum_{i,j,k}Q_{i j k}\langle A^-_{i j}, \nabla_k^{\perp} \nu_1\rangle&\leq 2 |\langle\nabla^{\perp} A^-, \nu_1\rangle|^2+\left(2 a_2+2 a_3\frac{18\tilde{k}(n+2)}{(7n-18)(n-\tilde{k})}\right) \frac{|A^-|^2}{f}|\langle\nabla^{\perp} \mathring{A}, \nu_1\rangle|^2 \nonumber\\
	&+2|A^-|^2|\nabla^{\perp} \nu_1|^2+\frac{2}{a_2} f|\nabla^{\perp} \nu_1|^2+\frac{2}{a_3}|\mathring{h}|^2|\nabla^{\perp} \nu_1|^2.
	\end{align*}
Therefore, it only remains to compare the coefficients of like terms in the two inequalities above. Firstly, we need
	\begin{align*}
	&\frac{2}{a_3}\le\left(\frac{7n-18}{9(n+2)(nc_n-1)}-2\right)\Longleftrightarrow
a_3\ge \frac{18(n+2)(nc_n-1)}{7n-18-18(n+2)(nc_n-1)}.
	\end{align*}

In the same way, we have
	\begin{align*}
	a_2\ge\frac{18(n+2)(nc_n-1)}{7n-18}.
	\end{align*}
Also, by comparing the coefficients for the term $|A^-|^2|\nabla^\bot\nu_1|^2$, we need
	\begin{align*}
	&2\le\frac{n}{nc_n-1}\left(\frac{16}{9(n+2)}-c_n\right)+\frac{7n-18}{9(n+2)(nc_n-1)}
\Longleftrightarrow \tilde{k} \leq \frac{n(14 n-36)}{41 n+18},
	\end{align*}
which for $\tilde{k}=2$, we have $n\ge9$. Using these values for $a_2$ and $a_3$, for the coefficients of $\frac{|A^-|^2}{f}|\langle\nabla^{\perp} \mathring{A}, \nu_1\rangle|^2$, we need
	\begin{align*}
	&\left(2 a_2+2 a_3\frac{18\tilde{k}(n+2)}{(7n-18)(n-\tilde{k})}\right)\le(1-\delta)\left(2-\frac{18(n+2)\left(n c_n-1\right)}{7n-18}\right)\Longleftrightarrow\nonumber\\
	\delta&\le1-\frac{36(n+2)(nc_n-1)\left( (n-\tilde{k})\left(7n-18-18(n+2)(nc_n-1)\right)+18\tilde{k}(n+2)\right)}{(n-\tilde{k})\left(7n-18-18(n+2)(nc_n-1)\right)\left(14n-36-18(n+2)(nc_n-1)\right)}.
	\end{align*}
And for $\tilde{k}=2$, we have
	\begin{align*}
\delta&\le1-\frac{72(n+2)(7 n-18)}{(n-2)\left(7 n-18-36(n+2) \cdot \frac{2}{n-2}\right)\left(14 n-36-18(n+2) \cdot \frac{2}{n-2}\right)}.
	\end{align*}
which is between $0$ and $1$, for any $n\ge23$.
Therefore, we get the estimate of the lemma. This completes the proof.
\end{proof}
Let $\delta$ be sufficiently small so that each of our above calculations hold. We begin by splitting off the desired nonpositive term in the evolution equation. We have
	\begin{align*}
	&\left(\partial_t-\Delta\right)\frac{|A^-|^2}{f}=\frac{1}{f}\left(\partial_t-\Delta\right)|A^-|^2-|A^-|^2\frac{1}{f^2}\left(\partial_t-\Delta\right)f+2\left\langle\nabla\frac{|A^-|^2}{f},\nabla\log f\right\rangle\\
	&=2\left\langle\nabla\frac{|A^-|^2}{f},\nabla\log f\right\rangle-\delta\frac{|A^-|^2}{f^2}\left(\partial_t-\Delta\right)f+\frac{1}{f}\left(\partial_t-\Delta\right)|A^-|^2-(1-\delta)\frac{|A^-|^2}{f^2}\left(\partial_t-\Delta\right)f.
	\end{align*}
From the calculation in page 91 and all the Lemmas from section \ref{section8.2}, we have the following estimate.
	\begin{align*}
	&\frac{1}{f}\left(\partial_t-\Delta\right)|A^-|^2-(1-\delta)|A^-|^2\frac{1}{f^2}\left(\partial_t-\Delta\right)f\\
	&\le2(1-\delta)\frac{|A^-|^2}{f^2}\left(\frac{2}{nc_n-1}d_n(|A^-|^2+f)+\frac{nc_n}{nc_n-1}|\mathring{h}|^2d_n \right)\\
	&-\frac{1}{f}\left( (1-\delta)\frac{2n}{n c_n-1}\left(c_n-\frac{16}{9(n+2)}\right)\frac{|A^-|^2}{f}d_n|\nabla^{\perp} \nu_1|^2-\frac{7n-18}{9(n+2)\left(n c_n-1\right)}d_n|\nabla^\bot\nu_1|^2\right)\\
	&+\frac{1}{f}\left(\frac{7n-18}{9(n+2)\left(n c_n-1\right)}d_n|\nabla^\bot\nu_1|^2-\frac{2n}{n c_n-1}\left(c_n-\frac{16}{9(n+2)}\right)\frac{|A^-|^2}{f}d_n|\nabla^{\perp} \nu_1|^2\right)\\
	&\le C,
	\end{align*}for $C$ constant depending on $n, K_1, K_2$ and $d_n$. 
Thus, according to our previous calculations, we get \eqref{7initialclaim}, which was our initial claim.
	\begin{align*}
	\left(\partial_t-\Delta\right)\frac{|A^-|^2}{f}&\le2\left\langle\nabla\frac{|A^-|^2}{f},\nabla\log f\right\rangle-\delta\frac{|A^-|^2}{f^2}\left(\partial_t-\Delta\right)f+C.
	\end{align*}Recall $\left(\partial_t-\Delta\right)f$ is non negative at each point in space time.
Now,
	\begin{align*}
	\frac{1}{f}\left(\partial_t&-\Delta\right)|A^-|^2-\frac{|A^-|^2}{f^2}\left(\partial_t-\Delta\right)f= \left(\partial_t-\Delta\right)\frac{|A^-|^2}{f}-2\left\langle\nabla\frac{|A^-|^2}{f},\nabla\log f\right\rangle\\
	&\le-\delta\frac{|A^-|^2}{f^2}\left(\partial_t-\Delta\right)f+C'\frac{|A^-|^2}{f}\\
	&+\frac{1}{f}\left(2\sum_{i,j}|\bar{R}_{ij}(\nu_1)|^2+4\sum_{i,j,p}\langle\bar{R}_{ij}(\nu_1),\mathring{h}_{ip}A^-_{jp}-\mathring{h}_{jp}A^-_{ip}\rangle \right).
	\end{align*}
\begin{theorem}[Codimension estimate]\label{blowuptheorem2}
Let $F: \mathcal{M}^n\times[0, T) \rightarrow \mathbb{C}P^k$ be a smooth solution to MCF so that
$F_0\colon \mathcal{M} \rightarrow \mathbb{C}P$ is compact and is quadratically pinched as in \eqref{pinching_conditioncpk}, for $2\le m<\frac{n-3}{4}$.
Then, $\forall \e>0, \exists H_0 >0$, such that if $f \geq H_0$, then
	\begin{align*}
	|A^-|^2 \leq \e f+C_{\e}
	\end{align*}$\forall t \in[0, T)$ where $C_\e=C_{\e}(n, m)$.
\end{theorem}
\begin{proof}
Since $\mathcal{M}$ is quadratically bounded, there exist constants $ C,D$ such that $|A^-|^2 \leq Cf +D.$ Therefore, the above estimate holds for all $\e> \frac{c_n}{\delta}$. Indeed, from the pinching $|A|^2\leq c_n|H|^2+d_n$, we can make a little bit more space so that
	\begin{align*}
	|A^-|^2+|A^+|^2=|A|^2\le (c_n-\delta)|H|^2+d_n-C_\delta
	\end{align*}
and therefore,
	\begin{align*}
	\delta|H|^2&\le c_n|H|^2-|A|^2+d_n-C_\delta\le c_n|H|^2-|A|^2+d_n.
	\end{align*}But since $|A^-|^2\le|A|^2\le c_n|H|^2+d_n$, we have $\frac{\delta |A^-|^2}{c_n}\le\frac{\delta|A|^2}{c_n}\le\delta|H|^2+\frac{\delta d_n}{c_n}$, so
	\begin{align*}
	\frac{\delta}{c_n} |A^-|^2 \le c_n |H|^2+d_n -|A|^2+\frac{\delta d_n}{c_n}= f+\frac{\delta d_n}{c_n},
	\end{align*}
which means that $|A^-|^2\le\frac{c_n}{\delta}f+d_n\le\varepsilon f+C_\varepsilon$. Hence, let $\e_0$ denote the infimum of such $\e$ for which the estimate is true and suppose $\e_0>0$. We will prove the theorem by contradiction. Hence, let us assume that the conclusions of the theorem are not true. That is, there exists a family of mean curvature flow $\mathcal{M}_t^k$ with points $(p_k,t_k)$ such that
	\begin{align}\label{eqn_limitepsilon0}
	\lim_{k\rightarrow \infty} \frac{|A_k^-(p_k,t_k)|^2}{f_k(p_k,t_k)}= \e_0,
	\end{align}
with $\e_0>0$ and $ f_k(p_k,t_k)\rightarrow \infty$.
We perform a parabolic rescaling of $ \mathcal{M}_t^k $ in such a way that $f_k$ at $(p_k,t_k)$ becomes $1$. If we consider the exponential map $\exp_{\bar{p}}\colon T_{\bar{p}}\mathbb{C}P^k \cong \mathbb{R}^{n+m}\to \mathbb{C}P^k$ and $\gamma$ a geodesic, then for a vector $v\in T_{\bar{p}}\mathbb{C}P^k$, then
	\begin{align*}
	\exp_{\bar{p}}(v)=\gamma_{\bar{p},\frac{v}{|v|}} (|v|), \ \ \gamma '(0)=\frac{v}{|v|} \ \ \text{ and} \ \ \gamma(0)=\bar{p}=F_k(p_k,t_k).
	\end{align*}
That is, if $F_k$ is the parameterisation of the original flow $ \mathcal{M}_t^k $, we let $ \hat r_k = \frac{1}{f_k(p_k,t_k)}$, and we denote the rescaled flow by $ \overline{\mathcal{M}}_\tau^k $ and we define its parameterisation by
	\begin{align*}
	\overline F_k (p,\tau) = \exp^{-1}_{F_k(p_k,t_k)} \circ F_k (p,\hat{r}^2_k \tau+t_k).
	\end{align*}
In the Riemannian case, when we change the metric after dilation, we do not need to multiply the immersion by the same constant as we would do in the Euclidean space. When we rescale the background space, we see that $\bar{g}_{ij}=\frac{1}{\hat{r}^2_k}g_{ij}$ and $\overline{K}=\hat{r}^2_k K,$ where $K$ is the sectional curvature of $\mathcal{N}$. In the same way, $|\overline{A}|^2=\hat{r}^2_k |A|^2$ and $|\overline{H}|^2=\hat{r}^2_k |H|^2.$
Since $d_n$ depends on $n$ and the sectional curvature $K$, the new $\bar{d}_n$ depends on n and $\overline{K}$. Hence, $\bar{d}_n=\hat{r}^2_k d_n.$\\
For $\hat{r}_k\to 0$, the background Riemannian manifold will converge to its tangent plane in a pointed $C^{\hat{d},\gamma}$ H\"older topology \cite{Petersen2016}. Therefore, we can work on $\mathbb{C}P^k$ as we would work in a Euclidean space. For simplicity, we choose for every flow a local co-ordinate system centred at $ p_k$. In these co-ordinates we can write $0$ instead of $ p_k$. Recall \eqref{parneigh}. The parabolic neighbourhoods $\mathcal P^k ( p_k, t_k, \hat r_k L, \hat r_k^2 \theta)$ in the original flow becomes $ \overline{\mathcal P}^k(0,0,L,\theta)$. By construction, each rescaled flow satisfies
	\begin{align} \label{eqn_H1}
	\overline F_k (0,0) = 0, \quad \overline f_k (0,0) = 1.
	\end{align}
Indeed,
	\begin{align*}
	\overline F_k(0,0)&=\exp_{F_k(0,0)}^{-1}\circ F_k(0,\hat r_k^2 \cdot 0)=0,\\
	\overline f_k (p,\tau)&=-|\overline A_k (p,\tau)|^2+c_n|\overline{H}_k (p,\tau)|^2+\bar{d}_n\\
	&=\hat r_k^2\Big(-|A_k(p,\hat r_k^2 \tau+t_k)|^2+c_n|H_k(p,\hat r_k^2 \tau+t_k)|^2+d_n\Big)\\
	&=\hat r_k^2 f_k(p,\hat r_k^2 \tau+t_k)
	\end{align*}
and so
	\begin{align*}
	\overline f_k (0,0)=\hat r_k^2 f_k(0,0)=1,
	\end{align*}
since $\hat r_k =\frac{1}{f_k (0,0)}=1$ from the change of coordinates. The gradient estimates give us uniform bounds (depending only on the pinching constant) on $ |A_k|$ and its derivatives up to any order on a neighbourhood of the form $\overline{\mathcal P }^k( 0 ,0,d,d)$ for a suitable $ d > 0$. From Theorem \ref{7thm_gradient}, we obtain gradient estimates on the second fundamental form in $ C^\infty $ on $ \overline F_k$. Hence we can apply Arzela-Ascoli (via the Langer-Breuning compactness theorem \cite{Breuning2015} and \cite{Langer1985}) and conclude there exists a subsequence converging in $ C^\infty $ to some limit flow which we denote by $ \widetilde{\mathcal{M}}_\tau^\infty$. We analyse the limit flow $ \widetilde{\mathcal{M}}_\tau^\infty$. The notation with the tilde's will refer to the limit flow. Note we have for the Weingarten map
	\begin{align*}
		[\overline A_k^-]_i^j ( p , \tau) = \hat r_k [A_k^-]_i^j ( p , \hat r_k^2 \tau+t_k),
	\end{align*}
so that
	\begin{align*}
	\frac{|\overline A_k^-(p, \tau) |^2}{\overline{f}_k(p,\tau)}&= \frac{|A_k^- ( p, \hat r_k^2 \tau+t_k) |^2}{f_k(p, \hat r_k^2 \tau+t_k)}.
	\end{align*}
From \eqref{eqn_limitepsilon0} and \eqref{eqn_H1}, we see
	\begin{align*}
	\frac{| \widetilde A^-(0,0)|^2}{\widetilde f(0,0)}= \e_0, \quad \widetilde f(0,0) = 1.
	\end{align*}
We claim
	\begin{align*}
	\frac{|\widetilde A^- ( p , \tau) |^2}{\widetilde f (p, \tau)}=\lim_{k\rightarrow \infty}\frac{|\overline A_k^-(p,\tau)|^2}{\overline f_k (p,\tau)} \leq \e, \quad \forall \e>\e_0.
	\end{align*}
 Since $\widetilde{f}(0,0)=1$, it follows that $|\widetilde{f}| \geq \frac{1}{2}$ in $ \widetilde{\mathcal{P}}^\infty(0,0,r,r)$ for some $r < d^\#$, where $d^\#$ is the constant from [\cite{Nguyen2018}, Lemma 3.5].
This is true since any point $( p, \tau) \in \widetilde{\mathcal{M}}^\infty_\tau$ is the limit of points $(p_{j_k},t_{j_k}) \in \overline{\mathcal{M}}^k_\tau$ and for every $ \e > \e_0 $ if we let $ \eta = \eta(\e,c_n)< d^\#$ then for large $k$, $\overline{\mathcal{M}}^k_\tau$ is defined in
	\begin{align*}
	 \mathcal P^k\left(p_{j_k},t_{j_k}, \frac{1}{f_k(p_{j_k},t_{j_k})}\eta,\left(\frac{1}{f_k(p_{j_k},t_{j_k})}\right)^2 \eta\right),
	\end{align*}
which implies
	\begin{align*}
	\frac{|\overline A_k^- ( p_{j_k} , t_{j_k}) |^2}{\overline f_k (p_{j_k}, t_{j_k})} \leq \e, \quad \forall \e>\e_0.
	\end{align*}
Hence the flow $\widetilde{\mathcal{M} }_\tau^\infty \subset \mathbb R^{n+m}$ has a space-time maximum $\e_0$ for $\frac{|\widetilde A^- ( p , \tau) |^2}{\widetilde f (p, \tau)}$ at $ (0,0)$ and the metric is the extrinsic Euclidean metric. From \cite{Naff}, the evolution equation for $ \frac{|{\widetilde{A}}^-|^2}{{\widetilde{f}}}$ in Euclidean space is given by 
	\begin{align*}
	\left(\partial_t-\Delta\right)\frac{|\widetilde A^-|^2}{\widetilde f}&\le2\left\langle\nabla\frac{|\widetilde A^-|^2}{\widetilde f},\nabla\log \widetilde f\right\rangle-\delta\frac{|\widetilde A^-|^2}{\widetilde f^2}\left(\partial_t-\Delta\right)\widetilde f.
	\end{align*}
Hence, since $\frac{|\widetilde A^-|^2}{\widetilde f}$ attains a maximum $\e_0$ at $(0,0)$ by the strong maximum principle then there exists this constant $\mathcal{C}$, such that $\frac{|\widetilde A^-|^2}{\widetilde f}=\mathcal{C}.$ Putting this into the evolution equation, we have
	\begin{align*}
	0&\le-\delta\frac{\mathcal{C}}{\widetilde f}\left(\partial_t-\Delta\right)\widetilde f\le 0.
	\end{align*}
From the evolution equation of the pinching quantity $\widetilde{f}$ we know that $\left(\partial_t-\Delta\right)\widetilde f\geq 0$ if $ \widetilde{f}\geq0$. And from the estimates in \cite{AnBa10}, \cite{HTNsurgery} and \cite{Naff} it is well known that $\left(\partial_t-\Delta\right)\widetilde f=0$ implies that $|\widetilde A^-|^2=0$ in which case $\e_0=0$. In the case where $\left(\partial_t-\Delta\right)\widetilde f>0$ we see that $\mathcal{C} =0$.

This implies $\frac{|\widetilde A^-|^2}{\widetilde{f}}=0\implies \e_0=0,$ which is a contradiction. Hence, we obtain
	\begin{align*}
	\lim_{k\rightarrow \infty} \frac{|\overline A_k^-(p_k,t_k)|^2}{\overline f_k(p_k,t_k)} = 0.
	\end{align*}
\end{proof}
\subsection{Cylindrical estimate}
In this section, we present estimates that demonstrate an improvement in curvature as we approach a singularity. These estimates play a critical role in the analysis of high curvature regions in geometric flows. In particular, in the high codimension setting, we establish that the quadratic pinching ratio $\frac{|A|^2}{|H|^2}$ approaches the ratio of a $\mathbb{S}^{n-l+1}\times\mathbb{R}^{l-1}$ cylinder, which is $\frac{1}{n-l+1}$, for $l$ positive integer.
\begin{theorem}[Cylindrical estimate]\label{thm_cylindrical2}
Let $F\colon \mathcal{M}^n\times[0, T) \rightarrow \mathbb{C}P^k$ be a smooth solution to mean curvature flow so that $F_0\colon \mathcal{M} \rightarrow \mathbb{C}P$ is compact and quadratically pinched as in \eqref{pinching_conditioncpk}, with constant $c_n=\frac{1}{n-\tilde{k}}$ and $\tilde{k}$ according to the hypothesis of Theorem \ref{maintheorem}.
Then $\forall \e>0, \exists H_1 >0$, such that if $f \geq H_1$, then
	\begin{align*}
	|A|^2- \frac{1}{n-\tilde{k}+1}|H|^2\leq \e |H|^2+C_{\e},
	\end{align*}
$\forall t \in[0, T)$, where $C_\e=C_{\e}(n)$.
\end{theorem}
\begin{proof}
The proof follows closely the proof of Theorem \ref{blowuptheorem2} and it doesn't make use of the codimension. Therefore, the same argument works using any of the conditions in \eqref{pinching_conditioncpk}. Since $\mathcal{M}$ is quadratically bounded, there exist constants $ C,D$ such that
	\begin{align*}
	|A|^2- \frac{1}{n-\tilde{k}+1}|H|^2 \leq C|H|^2+D.
	\end{align*}
Hence, let $\e_0$ denote the infimum of such $\e$ for which the estimate is true and suppose $\e_0>0$. We will prove the theorem by contradiction. Hence, let us assume that the conclusions of the theorem are not true that is there exists a family of mean curvature flow $\mathcal{M}_t^k$ with points $(p_k,t_k)$, such that
	\begin{align}\label{eqn_limitepsilon0versiontwo}
	\lim_{k\rightarrow \infty} \frac{\left (|A_k(p_k,t_k)|^2- \frac{1}{n-\tilde{k}+1}|H_k(p_k,t_k)|^2\right )}{f_k(p_k,t_k)}= \e_0,
	\end{align}
with $\e_0>0$ and $ f_k(p_k,t_k)\rightarrow \infty$.	

We perform a parabolic rescaling of $ \mathcal{M}_t^k $ exactly as in Theorem \ref{blowuptheorem2}, which is in such a way that $f_k$ at $(p_k,t_k)$ becomes $1$. If we consider the exponential map $\exp_{\bar{p}}\colon T_{\bar{p}}\mathbb{C}P^k \cong \mathbb{R}^{n+k}\to \mathbb{C}P^k$ and $\gamma$ a geodesic, then for a vector $v\in T_{\bar{p}}\mathbb{C}P^k$, we have
	\begin{align*}
	\exp_{\bar{p}}(v)=\gamma_{\bar{p},\frac{v}{|v|}} (|v|), \ \ \gamma '(0)=\frac{v}{|v|} \ \ \text{ and} \ \ \gamma(0)=\bar{p}=F_k(p_k,t_k).
	\end{align*}
That is, if $F_k$ is the parameterisation of the original flow $ \mathcal{M}_t^k $, we let $ \hat r_k = \frac{1}{f_k(p_k,t_k)}$, and we denote the rescaled flow by $ \overline{\mathcal{M}}_\tau^k $ and we define its parameterisation by
	\begin{align*}
	\overline F_k (p,\tau) = \exp^{-1}_{F_k(p_k,t_k)} \circ F_k (p,\hat{r}^2_k \tau+t_k).
	\end{align*}
For $\hat{r}_k\to 0$, the background Riemannian manifold will converge to its tangent plane in a pointed $C^{d,\gamma}$ H\"older topology \cite{Petersen2016}. Therefore, we can work on the manifold $\mathbb{C}P^k$ as we would work in a Euclidean space. For simplicity, we choose for every flow a local co-ordinate system centred at $ p_k$. In these co-ordinates we can write $0$ instead of $ p_k$. The parabolic neighbourhoods $\mathcal P^k ( p_k, t_k, \hat r_k L, \hat r_k^2 \theta)$ in the original flow becomes $ \overline{\mathcal P}^k(0,0,L,\theta)$. By construction, each rescaled flow satisfies
	\begin{align} \label{eqn_H1version2}
	\overline F_k (0,0) = 0, \quad \overline f_k (0,0) = 1.
	\end{align}
The gradient estimates give us uniform bounds (depending only on the pinching constant) on $ |A_k|$ and its derivatives up to any order on a neighbourhood of the form $\overline{\mathcal P }^k( 0 ,0,d,d)$ for a suitable $ d > 0$. From Theorem \ref{7thm_gradient}, we obtain gradient estimates on the second fundamental form in $ C^\infty $ on $ \overline F_k$. Hence we can apply Arzela-Ascoli (via the Langer-Breuning compactness theorem \cite{Breuning2015} and \cite{Langer1985}) and conclude there exists a subsequence converging in $ C^\infty $ to some limit flow which we denote by $ \widetilde{\mathcal{M}}_\tau^\infty$. 

From \eqref{eqn_limitepsilon0versiontwo} and \eqref{eqn_H1version2}, we see
	\begin{align*}
	\frac{| \widetilde A(0,0)|^2-\frac{1}{n-\tilde{k}+1}|\widetilde H(0,0)|^2}{\widetilde f(0,0)}= \e_0, \quad \widetilde f(0,0) = 1.
	\end{align*}
We claim
	\begin{align*}
	\frac{|\widetilde A( p , \tau) |^2-\frac{1}{n-\tilde{k}+1}|\widetilde H ( p , \tau) |^2}{\widetilde f (p, \tau)}=\lim_{k\rightarrow \infty}\frac{|\overline A_k(p,\tau)|^2-\frac{1}{n-\tilde{k}+1}|\overline H_k(p,\tau)|^2}{\overline f_k (p,\tau)} \leq \e \quad \forall \e>\e_0.
	\end{align*}
 Since $\widetilde{f}(0,0)=1$, it follows that $|\widetilde{f}| \geq \frac{1}{2}$ in $ \widetilde{\mathcal{P}}^\infty(0,0,r,r)$ for some $r < d^\#$.
This is true since any point $( p, \tau) \in \widetilde{\mathcal{M}}^\infty_\tau$ is the limit of points $(p_{j_k},t_{j_k}) \in \overline{\mathcal{M}}^k_\tau$ and for every $ \e > \e_0 $ if we let $ \eta = \eta(\e,c_n)< d^\#$ then for large $k$, $\mathcal{\overline{M}}^k_t$ is defined in
	\begin{align*}
	 \mathcal P^k\left(p_{j_k},t_{j_k}, \frac{1}{f_k(p_{j_k},t_{j_k})}\eta,\left(\frac{1}{f_k(p_{j_k},t_{j_k})}\right)^2 \eta\right),
	\end{align*}
which implies
	\begin{align*}
	\frac{|\overline A_k ( p, \tau) |^2-\frac{1}{n-\tilde{k}+1}|\overline H_k( p, \tau) |^2}{\overline f_k (p, \tau)} \leq \e \quad \forall \e>\e_0.
	\end{align*}
Hence, the flow $\widetilde{\mathcal{M} }_\tau^\infty \subset \mathbb R^{n+k}$ has a space-time maximum $\e_0$ for $\frac{|\widetilde A( p , \tau) |^2-\frac{1}{n-\tilde{k}+1}|\widetilde H ( p , \tau) |^2}{\widetilde f (p, \tau)}$ at $ (0,0)$, which implies that the flow $\widetilde{\mathcal M }_\tau^\infty$ has a space-time maximum $\frac{1}{n-\tilde{k}+1}+\e_0$ for $\frac{|\widetilde A ( p , \tau) |^2}{| \widetilde H (p, \tau) |^2}$ at $ (0,0)$. From now on, we drop the $ \ \widetilde{} \ $-sign for simplicity. Since the evolution equation for $ \frac{|A|^2}{|H|^2}$ is given by
	\begin{align*}
	\Big(\partial_t -\Delta\Big) \frac{|A|^2}{|H|^2}&= \frac{2}{|H|^2}\left\la \nabla |H|^2 , \nabla \left( \frac{|A|^2}{|H|^2}\right) \right\ra-\frac{2}{|H|^2} \left( |\nabla A|^2-\frac{|A|^2}{|H|^2}|\nabla H|^2 \right) \\
	&+\frac{2}{|H|^2}\left( R_1-\frac{|A|^2}{|H|^2} R_2\right)
	\end{align*}
and knowing that $\frac{3}{n+2}|\nabla H|^2 \leq |\nabla A|^2$ and $\frac{|A|^2}{|H|^2}\leq c_n,$ we arrive at
	\begin{align*}
	-\frac{2}{|H|^2} \left( |\nabla A|^2-\frac{|A|^2}{|H|^2}|\nabla H|^2 \right) \leq 0.
	\end{align*}
Furthermore, if $\frac{|A|^2}{|H|^2}=c < c_n$, for $\varepsilon_0$ small enough, according to Lemma 2.3 in \cite{HTNsurgery} and Theorem \ref{ThCPn}, we have
	\begin{align*}
	& R_1-\frac{|A|^2}{|H|^2} R_2= R_1-c R_2\\
	&\leq \frac{2}{n}\frac{1}{c-\nicefrac{1}{n}}| A^-|^2 \textsl{g}+\left(6-\frac{2}{n (c-\nicefrac{1}{n})} \right) |\circo h|^2 | \circo A^-|^2+\left(3-\frac{2}{n (c-\nicefrac{1}{n})} \right)|\circo A^-|^4+2\textsl{g} |h|^2\\
	&\leq 0,
	\end{align*}
where $\textsl{g}=-f$. Hence, the strong maximum principle applies to the evolution equation of $\frac{|A|^2}{|H|^2}$ and shows that $\frac{|A|^2}{|H|^2}$ is constant. The evolution equation then shows $ |\nabla A|^2= 0$, that is the second fundamental form is parallel and that $|A^-|^2 = |\circo A^-|^2=0$, that is the submanifold is codimension one. Finally, this shows locally $ \mathcal{M} = \mathbb S^{n-q}\times \mathbb R^q$, \cite{Lawson1969}. As $\frac{|A|^2}{|H|^2}< c_n\leq \frac{1}{n-\tilde{k}}$, we can only have
	\begin{align*}
	\mathbb S^n, \mathbb S^{n-1}\times \mathbb R, \ldots, \mathbb S^{n-\tilde{k}+1}\times \mathbb R^{\tilde{k}-1},
	\end{align*}
which gives $\frac{|A|^2}{|H|^2}= \frac{1}{n}, \frac{1}{n-\tilde{k}+1},\ldots,\frac{1}{n-\tilde{k}+1}\neq \frac{1}{n-\tilde{k}+1}+\e_0, \e_0>0$, which gives a contradiction.
\end{proof}
From Theorem \ref{blowuptheorem2} and Theorem \ref{thm_cylindrical2}, case (1) of Theorem \ref{maintheorem} is satisfied.
\subsection{Singularity models of quadratically pinched solutions of mean curvature flow in higher codimension}
Recall [Corollary 1.4, [\cite{Naff6}].  According to Theorem \ref{blowuptheorem2} and Theorem \ref{thm_cylindrical2}, for $F\colon \mathcal{M}^n\times[0, T) \rightarrow \mathbb{C}P^{k}$ be a smooth solution to mean curvature flow so that $F_0\colon \mathcal{M} \rightarrow \mathbb{C}P$ is compact and quadratically pinched as in \eqref{pinching_conditioncpk}, $n\ge\max\{23, 4\tilde{k}\}$, for $2\le m<\frac{n-3}{4}$, with $c_n=\frac{1}{n-\tilde{k}}$, $\tilde{k}=2$, then $\forall \e>0, \exists H_0, H_1 >0$, such that if $f \geq \max\{H_0,H_1\}$, then
	\begin{align*}
	|A^-|^2 \leq \e f+C_{\e} \quad \text{and} \quad  |A|^2- \frac{1}{n-\tilde{k}+1}|H|^2\leq \e |H|^2+C_{\e},
	\end{align*}
$\forall t \in[0, T)$, where $C_\e=C_{\e}(n, m)$. At the first singular time, the only possible blow-up limits are codimension one shrinking round spheres, shrinking round cylinders, and translating bowl solitons. Therefore, we can classify these blowup limits as follows.
\begin{corollary}\label{classificationcpn}
Let $c_n=\frac{1}{n-\tilde{k}}, \tilde{k}=2$, where $n$ satisfies the hypothesis in Theorem \ref{maintheorem}. Suppose $F_t\colon\mathcal{M}^n \rightarrow \mathbb{C}P^{k}$, $2\le m<\frac{n-3}{4}$, is a smooth solution of the mean curvature flow, compact and quadratically pinched as in \eqref{preservedcondition}, with positive mean curvature on the maximal time interval $[0, T)$.
\begin{enumerate}
\item If the singularity for $t \rightarrow T$ is of type I, the only possible limiting flows under the rescaling procedure as in \cite{HS}, are the homothetically shrinking solutions associated with $\mathbb{S}^n, \mathbb{R} \times \mathbb{S}^{n-1}$.
\item If the singularity is of type II, then from Theorem \ref{blowuptheorem2} and Theorem \ref{7thm_gradient}, the only possible blow-up limits at the first singular time are codimension one shrinking round spheres, shrinking round cylinders, and translating bowl solitons.
\end{enumerate}
\end{corollary}

\section{Convergence for infinite time}
In this section, we prove case (2) of Theorem \ref{maintheorem}. We assume $T= \infty$. We primarily follow \cite{LaLyNg}. In the following result, we prove a decay estimate on the traceless part of the second fundamental form, proving that the submanifold $\mathcal{M}_t$ converges to a totally geodesic limit, as $t\to T$. 
\subsection{Decay estimate}
\begin{proposition}[Decay estimate]\label{theoreminfinitetime}
Let $F: \mathcal{M}^n\times[0, T) \rightarrow \mathbb{C}P^k$ be a smooth solution to MCF so that $F_0\colon \mathcal{M} \rightarrow \mathbb{C}P$ is compact, quadratically pinched as in \eqref{pinching_conditioncpk}, and $n\ge4\tilde{k}$, where $\tilde{k}\ge1$. Then, there exists a positive constant $\mathcal{C}=\mathcal{C}(n,\tilde{k},d_n)$, depending only on the initial manifold $M_0$, such that
	\begin{align*}
	\frac{|\mathring{A}|^2}{f}\le\mathcal{C}e^{-2d_nt},
	\end{align*}
for any $0\le t<T=\infty$.
\end{proposition}
\begin{proof}
The proof doesn't make use of the codimension. Therefore, the same argument works using any of the conditions in \eqref{pinching_conditioncpk}. Consider the following functions
	\begin{align*}
Q :=-\frac{f}{2},\ \ \text{and} \ \ q:=\frac{1}{2}\left(|A|^2-\frac{1}{n}|H|^2\right),
	\end{align*}
where recall $f:=-|A|^2+c_n|H|^2+d_n$, $ c_n=\frac{1}{n-\tilde{k}}$ and $d_n=\frac{\tilde{k}}{n}(n-3-4m), 1\le\tilde{k}\le \frac{n}{4}$. In the case where $H\neq 0$, from \eqref{eqn_|A|^2}, \eqref{eqn_|H|^2} and as we did in previous chapters in \eqref{eqn_g} and \eqref{P_aestimate}, the evolution equation of $Q$ becomes
	\begin{align*}
	\left(\partial_t-\Delta\right) Q&=-\left(|\nabla A|^2-c_n|\nabla H|^2\right)+R_1-c_n R_2+\frac{1}{2}P\nonumber\\
&\le -\left(|\nabla A|^2-c_n|\nabla H|^2\right) +\sum_{\alpha,\beta}\left(\sum_{i,j} A^\alpha_{ij} A_{ij}^\beta \right)^2+\sum_{i,j,\alpha,\beta}\left(\sum_p\left(A^\alpha_{ip}A_{jp}^\beta -A^\alpha_{jp}A_{ip}^\beta \right)\right)^2\nonumber\\
&-c_n\sum_{i,j}\left(\sum_{\alpha} H^\alpha A^\alpha_{ij}\right)^2-(n-3-4m)|\mathring{A}|^2.
	\end{align*}
At a point where $H \neq 0$, decomposing $A$ into its irreducible components according to \cite{AnBa10}, \cite{LaLyNg}, \cite{Li1992} and \cite{Naff}, we have
	\begin{align*}
& |A|^2=|\mathring{h}|^2+\frac{1}{n} |H|^2+|A^-|^2, \\
&\sum_{i,j}\left(\sum_{\alpha} H^\alpha A^\alpha_{ij}\right)^2=|\mathring{h}|^2 |H|^2+\frac{1}{n} |H|^4, \\
&\sum_{\alpha,\beta}\left(\sum_{i,j} A^\alpha_{ij} A_{ij}^\beta \right)^2+\sum_{i,j,\alpha,\beta}\left(\sum_p\left(A^\alpha_{ip}A_{jp}^\beta -A^\alpha_{jp}A_{ip}^\beta \right)\right)^2\le 3|\mathring{h}|^2|A^-|^2+\frac{3}{2}|A^-|^4\\
& \ \quad\quad\quad\quad\quad\quad\quad\quad\quad\quad\quad\quad\quad\quad\quad\quad\quad\quad\quad\quad\quad\quad+\left(|\mathring{h}|^2+\frac{1}{n}|H|^2\right)|A|^2-\frac{1}{n}|A^-|^2|H|^2.
	\end{align*}
Therefore, we arrive at
	\begin{align*}
	\left(\partial_t-\Delta\right)Q&\le 3|\mathring{h}|^2|A^-|^2+\frac{3}{2}|A^-|^4+\left(|\mathring{h}|^2+\frac{1}{n}|H|^2\right)|A|^2-\frac{1}{n}|A^-|^2|H|^2\\
&-c_n\left(|\mathring{h}|^2+\frac{1}{n}|H|^2\right) |H|^2-\left(|\nabla A|^2-c_n|\nabla H|^2\right)-(n-3-4m)|\mathring{A}|^2\\
&=3|\mathring{h}|^2|A^-|^2+\frac{3}{2}|A^-|^4-\frac{1}{n}|A^-|^2|H|^2+\left(|\mathring{h}|^2+\frac{1}{n}|H|^2\right)(2Q+d_n)\\
&-(n-3-4m)(|\mathring{h}|^2+|A^-|^2)-\left(|\nabla A|^2-c_n|\nabla H|^2\right)\\
&=\left(3|\mathring{h}|^2+\frac{3}{2}|A^-|^2-\frac{1}{n}|H|^2-d_n\right)|A^-|^2+d_n\left(|\mathring{h}|^2+|A^-|^2+\frac{1}{n}|H|^2\right)\\
&-(n-3-4m)(|\mathring{h}|^2+|A^-|^2)+2Q\left(|\mathring{h}|^2+\frac{1}{n}|H|^2\right)-\left(|\nabla A|^2-c_n|\nabla H|^2\right).
	\end{align*}
Substituting $|A|^2=2Q+c_n|H|^2+d_n$, we obtain
	\begin{align*}
&d_n\left(|\mathring{h}|^2+|A^-|^2+\frac{1}{n}|H|^2\right)=d_n(2Q+c_n|H|^2+d_n),\\
&-(n-3-4m)|\mathring{A}|^2=-(n-3-4m)(|\mathring{h}|^2+|A^-|^2)\\
& \ \ \quad\quad\quad\quad\quad\quad\quad\quad=-(n-3-4m)\left(2Q+\left(c_n-\frac{1}{n}\right)|H|^2+d_n\right)
	\end{align*}
and hence
	\begin{align*}
&d_n\left(|\mathring{h}|^2+|A^-|^2+\frac{1}{n} |H|^2\right)  -(n-3-4m)(|\mathring{h}|^2+|A^-|^2) \\
& =Q\left(1-\frac{n}{\tilde{k}}\right)2d_n+|H|^2\left(c_n-\frac{n}{\tilde{k}}\left(c_n-\frac{1}{n}\right)\right)d_n+d_n^2\left(1-\frac{n}{\tilde{k}}\right).
	\end{align*}
Also, writing
	\begin{align*}
\frac{1}{n}|H|^2=\frac{1}{nc_n-1}(|\mathring{h}|^2+|A^-|^2-2Q-d_n),
	\end{align*}
we find
	\begin{align*}
3|\mathring{h}|^2&+\frac{3}{2}|A^-|^2-\frac{1}{n}|H|^2-d_n=3|\mathring{h}|^2+\frac{3}{2}|A^-|^2-\frac{1}{nc_n-1}(|\mathring{h}|^2+|A^-|^2-2Q-d_n)-d_n\\
&=\left(3-\frac{1}{nc_n-1}\right)|\mathring{h}|^2+\left(\frac{3}{2}-\frac{1}{nc_n-1}\right)|A^-|^2+\frac{2}{nc_n-1}Q-\left(1-\frac{1}{nc_n-1}\right)d_n.
	\end{align*}
For $c_n=\frac{1}{n-\tilde{k}}$, the first term on the right hand side is non positive, since $n\ge4\tilde{k}$, an inequality also used in Proposition 3.1 in \cite{LaLyNg} and the term $-\left(|\nabla A|^2-c_n|\nabla H|^2\right)$ is negative from \eqref{katoinequality2}. Disregarding these terms and putting things back together we conclude that
	\begin{align*}
&\left(\partial_t-\Delta\right)Q\le\left(\frac{3}{2}-\frac{1}{nc_n-1}\right)|A^-|^4-\left(1-\frac{1}{nc_n-1}\right)d_n|A^-|^2+\left(1-\frac{n}{\tilde{k}}\right)d_n^2\\
&+2Q\left(|\mathring{h}|^2+\frac{1}{nc_n-1}|A^-|^2+\frac{1}{n}|H|^2+d_n\left(1-\frac{n}{\tilde{k}}\right)\right)+d_n|H|^2\left(c_n-\frac{n}{\tilde{k}}\left(c_n-\frac{1}{n}\right)\right).
	\end{align*}
The term $d_n|H|^2\left(c_n-\frac{n}{\tilde{k}}\left(c_n-\frac{1}{n}\right)\right)$ is zero. Indeed,
	\begin{align*}
c_n-\frac{n}{\tilde{k}}\left(c_n-\frac{1}{n}\right)&=\frac{1}{n-\tilde{k}}-\frac{n}{\tilde{k}}\frac{1}{n-\tilde{k}}+\frac{1}{\tilde{k}}=0.
	\end{align*}
Also, the discriminant of the polynomial
	\begin{align*}
\left(\frac{3}{2}-\frac{1}{nc_n-1}\right)|A^-|^4-\left(1-\frac{1}{nc_n-1}\right)d_n|A^-|^2+\left(1-\frac{n}{\tilde{k}}\right)d_n^2
	\end{align*}
is negative. Indeed,
	\begin{align*}
D&=\left(1-\frac{1}{nc_n-1}\right)^2-4\left(\frac{3}{2}-\frac{1}{nc_n-1}\right)\left(1-\frac{n}{\tilde{k}}\right)\\
&=-5+\frac{2}{nc_n-1}+\frac{1}{(nc_n-1)^2}-\frac{4n}{\tilde{k}(nc_n-1)}+\frac{6n}{\tilde{k}}=-\frac{3n^2}{\tilde{k}^2}+\frac{10n}{\tilde{k}}-6<0,
	\end{align*}
for all $n\ge4\tilde{k}$ and $c_n=\frac{1}{n-\tilde{k}}$, so we can disregard these terms. Therefore, whenever $H\neq 0$,  the evolution equation of $Q$, becomes
	\begin{align*}
\left(\partial_t-\Delta\right)Q&<2Q\left(|\mathring{h}|^2+\frac{1}{nc_n-1}|A^-|^2+\frac{1}{n}|H|^2+d_n\left(1-\frac{n}{\tilde{k}}\right)\right).
	\end{align*}
For the evolution equation of $q$, we have
	\begin{align*}
\left(\partial_t-\Delta\right)q&=\left(3|\mathring{h}|^2+\frac{3}{2}|A^-|^2-\frac{1}{n}|H|^2\right)|A^-|^2-(n-3-4m)(|\mathring{h}|^2+|A^-|^2)\\
&-\left(|\nabla A|^2-\frac{1}{n}|\nabla H|^2\right)+2q\left(|\mathring{h}|^2+\frac{1}{n}|H|^2\right).
	\end{align*}
For $|\mathring{h}|^2+|A^-|^2=|\mathring{A}|^2=2q$, we have
	\begin{align*}
3|\mathring{h}|^2+\frac{3}{2}|A^-|^2-\frac{1}{n}|H|^2&=\left(3-\frac{1}{nc_n-1}\right)|\mathring{h}|^2+\left(\frac{3}{2}-\frac{1}{nc_n-1}\right)|A^-|^2-\frac{1}{n}|H|^2+\frac{2}{nc_n-1}q.
	\end{align*}
The first three terms of the above equality are non-positive. Also, from \eqref{katoinequality2}, we arrive at
	\begin{align*}
\left(\partial_t-\Delta\right)q&\le 2q\left(|\mathring{h}|^2+\frac{1}{nc_n-1}|A^-|^2+\frac{1}{n}|H|^2-\frac{n}{\tilde{k}}d_n\right).
	\end{align*}
Finally, for the evolution equation of $\frac{\left(\partial_t-\Delta\right)\frac{q}{Q}}{\frac{q}{Q}}$, we have the following computation
	\begin{align*}
\frac{\left(\partial_t-\Delta\right) \frac{q}{Q}}{\frac{q}{Q}} & =\frac{\left(\partial_t-\Delta\right) q}{q}-\frac{\left(\partial_t-\Delta\right) Q}{Q}+2\left\langle\nabla \log \frac{q}{Q}, \nabla \log Q\right\rangle \\
&\le2\left(|\mathring{h}|^2+\frac{1}{nc_n-1}|A^-|^2+\frac{1}{n}|H|^2-\frac{n}{\tilde{k}}d_n\right)\\
&-2\left(|\mathring{h}|^2+\frac{1}{nc_n-1}|A^-|^2+\frac{1}{n}|H|^2+d_n\left(1-\frac{n}{\tilde{k}}\right)\right)+2\left\langle\nabla \log \frac{q}{Q}, \nabla \log Q\right\rangle\\
&=-2d_n+2\left\langle\nabla \log \frac{q}{Q}, \nabla \log Q\right\rangle,
	\end{align*}
whenever $H \neq 0$. In the case of $H=0$, in the same way, we have
 	\begin{align*}
&\left(\partial_t-\Delta\right)Q\le2Q\left(|\mathring{h}|^2+\frac{1}{nc_n-1}|A^-|^2+d_n\left(1-\frac{n}{\tilde{k}}\right)\right),\\
&\left(\partial_t-\Delta\right)q\le 2q\left(|\mathring{h}|^2+\frac{1}{nc_n-1}|A^-|^2-\frac{n}{\tilde{k}}d_n\right).
	\end{align*}
Therefore, for the evolution equation of $\frac{\left(\partial_t-\Delta\right)\frac{q}{Q}}{\frac{q}{Q}}$, in the case where $H=0$, we have the following computation
	\begin{align*}
\frac{\left(\partial_t-\Delta\right) \frac{q}{Q}}{\frac{q}{Q}} 
&\le-2d_n+2\left\langle\nabla \log \frac{q}{Q}, \nabla \log Q\right\rangle,
	\end{align*}
which is to the same evolution equation as in the case $H\neq0$. Hence, by the strong maximum principle there exists this constant $\mathcal{C}$ depending upon $\tilde{k},n$ and $d_n$, such that
	\begin{align*}
	\frac{q}{Q}\le\mathcal{C}e^{-2d_nt},
	\end{align*}
which completes the proof.
\end{proof}
Proposition \ref{theoreminfinitetime} implies, that there exists $\tau=\tau(n, \tilde{k})$, such that the inequality
	\begin{align*}
|A|^2-\frac{1}{n-1}|H|^2-\frac{n-3-4m}{n}<0
	\end{align*}
holds for every $t \in(0, T) \cap\left[\tau d_n^{-1}, +\infty\right)$ on any solution initially satisfying \eqref{preservedcondition}. Therefore, for all $\tilde{k}\ge1$, we have 
	\begin{align*}
|A|^2-\frac{1}{n-\tilde{k}}|H|^2-\frac{\tilde{k}}{n}(n-3-4m)<0.
	\end{align*}
If $T>\tau d_n^{-1}$, this means that at time $\tau d_n^{-1}$ the solution satisfies the hypotheses of Theorem 1.1 in \cite{PipSin}. From Proposition \ref{theoreminfinitetime}, case (2) of Theorem \ref{maintheorem} is satisfied. Consequently, the solution either exists forever and converges to a totally geodesic submanifold as $t \rightarrow \infty$, or else contracts to codimension one solution in finite time, from Theorem \ref{blowuptheorem2}.


\begin{thebibliography}{10}



\bibitem{AnBa10}
Ben Andrews and Charles Baker.
\newblock Mean curvature flow of pinched submanifolds to spheres.
\newblock {\em J. Differential Geom.}, 85(3):357--395, 2010.

\bibitem{BakerThesis}
Charles Baker.
\newblock {\em The mean curvature flow of submanifolds of high codimension}.
\newblock Australian National University, 2011.
\newblock Thesis (Ph.D.)--Australian National University.

\bibitem{Besse}
Arthur L. Besse.
\newblock  Manifolds all of whose geodesics are closed.
\newblock {\em Springer-Verlag, Berlin, Heidelberg, New York, 1978.}

%
%

\bibitem{Breuning2015}
Patrick Breuning.
\newblock Immersions with bounded second fundamental form.
\newblock {\em J. Geom. Anal.}, 25(2):1344--1386, 2015.

%

\bibitem{Goldberg}
Samuel I. Goldberg.
\newblock Curvature and Homology, Academic Press, London, 1962.

\bibitem{Grayson89}
Matthew A. Grayson.
\newblock Shortening Embedded Curves.
\newblock {\em The Annals of Mathematics}, Vol. 129, No. 1, (Jan., 1989), pp. 71-111.

%

\bibitem{HK2}
Robert Haslhofer and Bruce Kleiner.
\newblock Mean curvature flow with surgery.
\newblock {\em Duke Math. J.}, 166(9):1591--1626, 2017.

\bibitem{Hu84}
Gerhard Huisken.
\newblock Flow by mean curvature of convex surfaces into spheres.
\newblock {\em J. Differential Geom.}, 20(1):237--266, 1984.

\bibitem{Hu86}
Gerhard Huisken.
\newblock Contracting convex hypersurfaces in Riemannian manifolds by their mean curvature.
\newblock {\em Invent. Math.}, 84(3):463--480, 1986.

\bibitem{Hu87}
Gerhard Huisken.
\newblock Deforming hypersurfaces of the sphere by their mean curvature.
\newblock {\em Mathematische Zeitschrift}, 195, pages205–219, 1987.


\bibitem{HS}
Gerhard Huisken and Carlo Sinestrari.
\newblock Mean curvature flow singularities for mean convex surfaces.
\newblock {\em Calc. Var. Partial Differential Equations}, 8(1):1--14, 1999.

\bibitem{HuSi09}
Gerhard Huisken and Carlo Sinestrari.
\newblock Mean curvature flow with surgeries of two-convex hypersurfaces.
\newblock {\em Invent. Math.}, 175(1):137--221, 2009.


\bibitem{Langer1985}
Joel Langer.
\newblock A compactness theorem for surfaces with {$L_p$}-bounded second
  fundamental form.
\newblock {\em Math. Ann.}, 270(2):223--234, 1985.

\bibitem{LaLyNg}
Mat Langford, Stephen Lynch and Huy The Nguyen.
\newblock Quadratically pinched submanifolds of the sphere via mean curvature flow with surgery.
\newblock Preprint, \href{https://doi.org/10.48550/arXiv.2109.03651}{arXiv:2109.03651v1 [math.DG]}.

\bibitem{Lawson1969}
H.~Blaine Lawson, Jr.
\newblock Local rigidity theorems for minimal hypersurfaces.
\newblock {\em Ann. of Math. (2)}, 89:187--197, 1969.

\bibitem{Li1992}
An-Min Li and Jimin Li.
\newblock An intrinsic rigidity theorem for minimal submanifolds in a sphere.
\newblock {\em Arch. Math. (Basel)}, 58(6):582--594, 1992.


\bibitem{Liu}
Kefeng Liu, Hongwei Xu, and Entao Zhao.
\newblock Mean curvature flow of higher codimension in {R}iemannian manifolds.
\newblock arXiv:1204.0107v1 [math.DG], 2012.



\bibitem{Naff}
Keaton Naff.
\newblock A planarity estimate for pinched solutions of mean curvature flow.
\newblock {\em Duke Math. J.}, 171(2):443--482, 2022.

\bibitem{Naff6}
Keaton Naff.
\newblock Singularity models of pinched solutions of mean curvature flow in higher codimension.
\newblock {\em Journal für die reine und angewandte Mathematik}, 2023(794),
  2023.


\bibitem{Nguyen2018}
Huy~The Nguyen.
\newblock Cylindrical Estimates for High Codimension Mean Curvature Flow.
\newblock \href{https://arxiv.org/abs/1805.11808v1}{arXiv:1805.11808v1}, 2018.

\bibitem{HTNsurgery}
Huy The Nguyen.
\newblock High Codimension Mean Curvature Flow with Surgery.
\newblock Preprint, \href{https://doi.org/10.48550/arXiv.2004.07163}{arXiv:2004.07163v2 [math.DG]}.

\bibitem{HNAV}
Huy T. Nguyen and Artemis A. Vogiatzi.
\newblock Singularity Models for High Codimension Mean Curvature Flow in Riemannian Manifolds.
\newblock Preprint, \href{https://doi.org/10.48550/arXiv.2303.00414}{arXiv:2303.00414v1 [math.DG]}.


\bibitem{Petersen2016}
Peter Petersen.
\newblock {\em Riemannian geometry}, volume 171 of {\em Graduate Texts in
  Mathematics}.
\newblock Springer, Cham, third edition, 2016.

\bibitem{Pip}
Giuseppe Pipoli.
\newblock Mean curvature flow and Riemannian submersions.
\newblock{\em Geom Dedicata}, (2016) 184:67–81.

\bibitem{PipSin}
Giuseppe Pipoli and Carlo Sinestrari.
\newblock Mean curvature flow of pinched submanifolds of $\mathbb{CP}^n$.
\newblock {\em Comm. Anal. Geom.}, 25 (4): 799 - 846, 2017.

\bibitem{PipSin17}
Giuseppe Pipoli and Carlo Sinestrari.
\newblock Cylindrical estimates for mean curvature flow of hypersurfaces in CROSSes.
\newblock {\em Ann. Glob. Anal. Geom.}, 51 (2): 179-188, 2017.


%
%
%
%
%
%
%

\bibitem{White00}
Brian White.
\newblock The size of the singular set in mean curvature flow of mean-convex sets.
\newblock {\em Journal of the American Mathematical Society}, 13 (2000), 665-695.

%



\end{thebibliography}
\end{document}